\RequirePackage{fix-cm}
\documentclass[smallextended, envcountsame]{svjour3}
\smartqed
\journalname{}

\usepackage[unicode = false,
    pdftoolbar = true,
    colorlinks = true,
    linkcolor = blue,
    citecolor = blue,
    filecolor = black,
    urlcolor = blue,
    breaklinks = true]{hyperref}

\usepackage[utf8]{inputenc}
\usepackage{listings}  % For Code 
\usepackage{enumitem}  
\usepackage{amsmath}   % For Math, duh
\usepackage{amssymb}   % For mathsfrs
\usepackage{xcolor}    % For Code Color
\usepackage{graphicx}  
\usepackage{verbatim}
\usepackage{geometry}
\usepackage{mathtools}
\usepackage{float}
\usepackage{bm}
\usepackage[ruled]{algorithm2e}
\usepackage{subcaption}
\usepackage{placeins}
\usepackage{listings}
\usepackage[justification=centering]{caption}

\definecolor{gab}{HTML}{50c878}
\definecolor{pat}{HTML}{bf68ff}
\definecolor{sham}{HTML}{7faaff}

\allowdisplaybreaks[4]

\newcommand{\myalg}{\mathrm{QARSTA}}
\newtheorem{assump}{Assumption}

\newtheorem{defini}{Definition}
\newtheorem{exampl}{Example}
\newtheorem{lemm}{Lemma}
\newtheorem{rema}{Remark}
\newtheorem{theo}{Theorem}

\newcommand\numberthis{\addtocounter{equation}{1}\tag{\theequation}}
\newcommand{\appdiam}{\overline{\mathrm{diam}}}

\makeatletter
\newcommand*{\rom}[1]{\expandafter\@slowromancap\romannumeral #1@}
\makeatother

\SetKwIF{If}{ElseIf}{Else}{if}{}{else if}{else}{end if}%
\allowdisplaybreaks[4]
\SetArgSty{textnormal}

\title{$Q$-fully quadratic modeling and its application in a random subspace derivative-free method}

\author{Yiwen Chen\and Warren Hare\and Amy Wiebe}
\institute{Department of Mathematics, University of British Columbia, Kelowna, British Columbia, V1V 1V7, Canada.\\
This research is partially funded by the Natural Sciences and Engineering Research Council of Canada (cette recherche est partiellement financ\'ee par le Conseil de recherches en sciences naturelles et en g\'enie du Canada), Discover Grant \#2023-03555 and by the Mitacs Globalink Graduate Fellowship.\\
\email{yiwchen@student.ubc.ca, warren.hare@ubc.ca, amy.wiebe@ubc.ca}}
\date{\today}

\begin{document}

\maketitle

% REQUIRED
\begin{abstract}
Model-based derivative-free optimization (DFO) methods are an important class of DFO methods that are known to struggle with solving high-dimensional optimization problems.  Recent research has shown that incorporating random subspaces into model-based DFO methods has the potential to improve their performance on high-dimensional problems.  However, most of the current theoretical and practical results are based on linear approximation models due to the complexity of quadratic approximation models.  This paper proposes a random subspace trust-region algorithm based on quadratic approximations.  Unlike most of its precursors, this algorithm does not require any special form of objective function.  We study the geometry of sample sets, the error bounds for approximations, and the quality of subspaces.  In particular, we provide a technique to construct $Q$-fully quadratic models, which is easy to analyze and implement.  We present an almost-sure global convergence result of our algorithm and give an upper bound on the expected number of iterations to find a sufficiently small gradient.  We also develop numerical experiments to compare the performance of our algorithm using both linear and quadratic approximation models.  The numerical results demonstrate the strengths and weaknesses of using quadratic approximations.
\end{abstract}

\section{Introduction}
    Derivative-free optimization (DFO) methods are a class of optimization methods that do not use the derivatives of the objective or constraint functions \cite{audet2017derivative,conn2009introduction}.  In recent years, DFO has continued to gain popularity and has demonstrated effective use in solving optimization problems where the derivative information is expensive or impossible to obtain.  Such problems occur regularly in many fields of modern research.  For example, the objective or constraints may be given by computer simulation or laboratory experiments.  A recent survey of the applications of DFO methods can be found in \cite{alarie2021two}, in which the authors review hundreds of applications of DFO methods and particularly focus on energy, materials science, and computational engineering design.   More recently, with the emergence of artificial intelligence, large-scale DFO methods have gained attention as they have great potential to handle problems such as hyperparameter tuning \cite{feurer2019hyperparameter,ghanbari2017black,li2021survey} and generating attacks to deep neural networks \cite{alzantot2019genattack,chen2017zoo}.

    One major class of DFO methods is model-based methods \cite{audet2020model,larson2019derivative,liuzzi2019trust}.  In model-based DFO methods, model functions are constructed iteratively to approximate the objective and constraint functions.  Unfortunately, most of the current model-based DFO methods are inefficient in high dimensions, since both the number of function evaluations and the linear algebra cost required to construct models increase with the problem dimension. 

    One promising technique to handle large optimization problems is through subspaces.  This technique has long been used in gradient-based optimization \cite{conn1994iterated,fukushima1998parallel,grapiglia2013subspace} and more recently been applied in DFO \cite{audet2008parallel,zhang2012derivative}.  Audet et al. \cite{audet2008parallel} apply a parallel space decomposition technique and introduce an asynchronous parallel generalized pattern search algorithm in which each processor solves the optimization problem over only a subset of variables.  In \cite{zhang2012derivative}, Zhang proposes a subspace trust-region method based on NEWUOA \cite{powell2006newuoa,powell2008developments}.  Numerical results demonstrate a significant improvement over NEWUOA on large-scale problems.  More recently, randomization is being incorporated by working in low-dimensional subspaces that are chosen randomly in each iteration \cite{cartis2023scalable,dzahini2022stochastic,hare2023expected,roberts2023direct}.  Roberts and Royer \cite{roberts2023direct} study a direct-search algorithm with search directions chosen in random subspaces.  In \cite{cartis2023scalable} and \cite{dzahini2022stochastic}, random subspace trust-region algorithms are proposed for solving nonlinear least-squares and stochastic optimization problems, respectively.  The theoretical and numerical benefits of random subspaces are confirmed by the reduced per-iteration cost and strong performance of these algorithms.  A theoretical analysis of the connection between the dimension of subspaces and the algorithmic performance can be found in~\cite{hare2023expected}.

    All of the above random subspace trust-region algorithms only construct linear approximation models.  This is because linear functions are easier to handle than quadratic functions in both theoretical analysis and computer programming aspects.  Indeed, recall that a determined linear interpolation model in $\mathbb{R}^n$ requires $(n+1)$ sample points to construct, whereas a determined quadratic interpolation model requires $(n+1)(n+2)\slash 2$ sample points \cite{conn2009introduction}.  This makes the geometry of sample sets, which is a key factor for model quality \cite{conn20081geometry,conn20082geometry}, more difficult to manage in the quadratic case.  Moreover, the only random subspace trust-region algorithm mentioned above developed for deterministic optimization is specialized to nonlinear least-squares problems.  The algorithm proposed in \cite{cartis2023scalable} achieves scalability by exploring the special structure of the objective functions in nonlinear least-squares problems and is therefore not able to handle optimization problems with general objective functions.
    
    This paper proposes a random subspace trust-region algorithm based on quadratic approximations for unconstrained deterministic optimization problems.  This algorithm does not require the objective function to have any specific form (e.g., nonlinear least-squares) and is therefore able to handle a much broader scope of optimization problems.  We provide rigorous theoretical analysis for model accuracy, sample set geometry, and subspace quality.  
    
    The sample set management procedure used in our algorithm is based on the condition number of the direction matrix corresponding to the sample set (see Subsection \ref{subsec:samplesetgeometry} for details).  This is different from the management procedures used in most of the previous algorithms, which are based on the Lagrange polynomials.  We note that our procedure is easier for managing quadratic interpolation sample sets than the procedure based on Lagrange polynomials from \cite{cartis2023scalable}.  This is because the Lagrange polynomial approach requires maximizing quadratic functions over balls.
    
    We prove the almost-sure global convergence of our algorithm and give a complexity bound on the expected number of iterations required to find a sufficiently small gradient.  By numerical experiments, we demonstrate the advantages and disadvantages of quadratic models over linear models and the benefits of exploring the structure of objective functions when possible.  
    
    It is worth mentioning that most of the theoretical results given in this paper are independent of the algorithm.  In particular, our $Q$-fully quadratic modeling technique is easy to handle for both theoretical analysis and implementation purposes and constructs a class of models that is more accurate than required to prove convergence (see Subsection~\ref{subsec:modelAcc} and Section \ref{sec:converg} for details).  Therefore, this paper also provides several theoretical tools for future DFO algorithm design.

    The remainder of this paper is organized as follows.  We end this section with remarks on our notation.  In Section \ref{sec:alg}, we describe the general framework of our algorithm and explain our model construction technique, sample set management procedure, and subspace selection strategy.  In Section \ref{sec:converg}, we provide the convergence analysis of our algorithm.  In Section \ref{sec:numexp}, we design numerical experiments to compare the performance of our algorithm based on quadratic and linear models.  Section \ref{sec:conclusion} concludes our work and gives some suggestions for future work.

\subsection{Notation and background definitions}
    We work in Euclidean space $\mathbb{R}^n$, with inner product $x^\top y=\sum_{i=1}^nx_iy_i$ and induced vector norm $\|x\|=\sqrt{x^\top x}$.  For a matrix $A\in\mathbb{R}^{n\times z}$, we use the induced matrix norm $\|A\|=\max\{\|Ax\|:x\in\mathbb{R}^z, \|x\|=1\}$.  The $ij$-th element of $A$ is denoted by $A_{ij}$. Functions $\lambda_{\min}(\cdot)$ and $\sigma_{\min}(\cdot)$ give the minimum eigenvalue and singular value of a matrix, respectively.  We use $A=[a_1\cdots a_z]$ to denote the column representation of $A$ and we write $a_i\in A$ to mean that vector $a_i$ is a column of $A$.  We use $I_n$ to denote the identity matrix of order $n$.  For $i\in\{1,...,n\}$, we use $e_i^{(n)}\in\mathbb{R}^n$ to denote the $i$-th column of $I_n$, i.e., the $i$-th coordinate vector in $\mathbb{R}^n$.  For a set $S$, we use $|S|$ to denote the cardinality of $S$.  The set $B(x;\Delta)$ is the closed ball centered at $x$ with radius $\Delta$: $$B(x;\Delta)=\left\{y:\left\|y-x\right\|\le\Delta\right\}.$$

    For an invertible square matrix $A\in\mathbb{R}^{n\times n}$, we denote the inverse of $A$ by $A^{-1}$.  Also, we use a generalization of the inverse matrix called the Moore–Penrose Pseudoinverse \cite{penrose1955generalized}.
    \begin{defini}(Moore–Penrose Pseudoinverse)
        For a matrix $A\in\mathbb{R}^{n\times z}$, the Moore–Penrose Pseudoinverse of $A$, denoted by $A^\dagger$, is the unique matrix in $\mathbb{R}^{z\times n}$ satisfying the following four conditions:
        \begin{enumerate}
            \item $AA^\dagger A=A,$
            \item $A^\dagger AA^\dagger=A^\dagger,$
            \item $(AA^\dagger)^\top=AA^\dagger,$
            \item $(A^\dagger A)^\top=A^\dagger A.$
        \end{enumerate}
    \end{defini}
    Note that $A^\dagger$ exists and is unique for any matrix $A\in\mathbb{R}^{n\times z}$ \cite{penrose1955generalized}.  In particular, when $A$ has full rank, $A^\dagger$ can be expressed by the following simple formulae \cite{penrose1955generalized}.
    \begin{enumerate}
        \item If $A\in\mathbb{R}^{n\times z}$ has full column rank $z$, then $A^\dagger$ is a left-inverse of $A$, i.e., $A^\dagger A=I_z$ and $$A^\dagger=(A^\top A)^{-1}A^\top.$$ 
        \item If $A\in\mathbb{R}^{n\times z}$ has full row rank $n$, then $A^\dagger$ is a right-inverse of $A$, i.e., $AA^\dagger=I_n$ and $$A^\dagger=A^\top(AA^\top)^{-1}.$$ 
    \end{enumerate}

    Now we introduce the generalized simplex gradient \cite{hare2020calculus} and the generalized simplex Hessian~\cite{hare2023matrix}.  We note that the generalized simplex Hessian we define below is a special case of the definition given in \cite{hare2023matrix}.
    \begin{defini}(Generalized simplex gradient)
        Let $x^0\in\mathbb{R}^n$ and $D=[d_1\cdots d_z]\in\mathbb{R}^{n\times z}$.  The generalized simplex gradient of $f$ at $x^0$ over $D$, denoted by $\nabla_S f(x^0;D)$, is defined by $$\nabla_S f(x^0;D) = \left(D^\top\right)^\dagger\delta_f(x^0;D),$$ where 
        \begin{equation}\label{eq:GSGdelta}
            \delta_f(x^0;D)=\begin{bmatrix}
            f(x^0 + d_1) - f(x^0)\\
            f(x^0 + d_2) - f(x^0)\\
            \vdots\\
            f(x^0 + d_z) - f(x^0)
        \end{bmatrix}.
        \end{equation}
    \end{defini}

    \begin{defini}(Generalized simplex Hessian)
        Let $x^0\in\mathbb{R}^n$ and $D=[d_1\cdots d_z]\in\mathbb{R}^{n\times z}$.  The generalized simplex Hessian of $f$ at $x^0$ over $D$, denoted by $\nabla_S^2 f(x^0;D)$, is defined by $$\nabla_S^2 f(x^0;D) = \left(D^\top\right)^\dagger\delta_{\nabla_S f}(x^0;D),$$ where 
        \begin{equation}\label{eq:GSHdelta}
            \delta_{\nabla_S f}(x^0;D)=\begin{bmatrix}
            \left(\nabla_S f(x^0 + d_1;D) - \nabla_S f(x^0;D)\right)^\top\\
            \left(\nabla_S f(x^0 + d_2;D) - \nabla_S f(x^0;D)\right)^\top\\
            \vdots\\
            \left(\nabla_S f(x^0 + d_z;D) - \nabla_S f(x^0;D)\right)^\top\\
        \end{bmatrix}.
        \end{equation}
    \end{defini}

\section{Quadratic approximation random subspace trust-region algorithm}\label{sec:alg}
    This section introduces our Quadratic Approximation Random Subspace Trust-region Algorithm ($\myalg$).  We first explain how the interpolation models are constructed in $\myalg$ and analyze their accuracy.  Then we outline the framework of $\myalg$ and discuss some important details.  The convergence of $\myalg$ is proved in Section \ref{sec:converg}.

    Throughout this section, we consider the objective function $f:\mathbb{R}^n\to\mathbb{R}$.

    \subsection{Model construction}\label{subsec:modelconstruct}
    The main idea of $\myalg$ is that, in each iteration, an interpolation model is constructed in a subspace of $\mathbb{R}^n$.  This reduces the number of function evaluations and the computation effort to construct each model.  Moreover, as pointed out in \cite{cartis2023scalable}, since we do not capture the behaviors of the objective function outside our subspaces, this also improves the scalability of our algorithm.

    In the $k$-th iteration of $\myalg$, where the current iterate is $x_k$, we consider a $p$-dimensional affine space defined by 
    \begin{equation*}
        \mathcal{Y}_k=\left\{x_k+D_k\widehat{s}:\widehat{s}\in\mathbb{R}^p\right\},
    \end{equation*}
    where $D_k\in\mathbb{R}^{n\times p}$ has full-column rank.  

    Then, we construct a quadratic model that interpolates $f$ in $\mathcal{Y}_k$.  There are two equivalent ways to interpret this model: either as an underdetermined quadratic interpolation model in $\mathbb{R}^n$ or as a determined quadratic interpolation model in $\mathbb{R}^p$.  In this subsection, we explain how the model is constructed and how to convert between these two viewpoints.

    The underdetermined quadratic model in $\mathbb{R}^n$ is based on the generalized simplex gradient and generalized simplex Hessian.
    \begin{defini}
        Given $D=[d_1\cdots d_p]\in\mathbb{R}^{n\times p}$ with full-column rank, $f:\mathbb{R}^n\to\mathbb{R}$, and $x^0\in\mathbb{R}^n$. The model $m:\mathbb{R}^n\to\mathbb{R}$ is defined by 
        \begin{align*}m(x)=f(x^0)+\left(2\nabla_S f(x^0;D)-\nabla_S f(x^0;2D)\right)^\top\left(x-x^0\right) +\frac{1}{2}\left(x-x^0\right)^\top\nabla_S^2 f(x^0;D)\left(x-x^0\right).
        \end{align*}
    \end{defini}
    \begin{theo}\label{thm:fullspaceunderd}
        The model $m(x)$ is an underdetermined quadratic interpolation model of $f$ in $\mathbb{R}^n$. 
        In particular, $m(x)$ interpolates $f(x)$ on all sample points in 
        \begin{equation*}
            Y = \left\{x^0\right\}\cup\left\{x^0+d_i:i=1,...,p\right\}\cup\left\{x^0+d_i+d_j:i,j=1,...,p\right\}.
        \end{equation*}
    \end{theo}
    \begin{proof}
        We only need to show that $m(x)=f(x)$ for all $x$ in the set $Y$.
        
        {\bf Case \rom{1}:} Suppose $x=x^0$.  Then we have
        \begin{equation*}
            m(x^0)=f(x^0)+\left(2\nabla_S f(x^0;D)-\nabla_S f(x^0;2D)\right)^\top\mathbf{0} +\frac{1}{2}\mathbf{0}^\top\nabla_S^2 f(x^0;D)\mathbf{0}=f(x^0).
        \end{equation*}

        {\bf Case \rom{2}:} Suppose $x=x^0+d_i$ where $i\in\{1,...,p\}$. Since $D$ has full-column rank, we have $D^\dagger D=I_p$. Therefore,
        \begin{equation*}
            \nabla_S f(x^0;D)^\top D = \delta_f(x^0;D)^\top D^\dagger D=\delta_f(x^0;D)^\top,
        \end{equation*}
        \begin{equation*}
            \nabla_S f(x^0;2D)^\top D = \delta_f(x^0;2D)^\top \left(2D\right)^\dagger D=\frac{1}{2}\delta_f(x^0;2D)^\top,
        \end{equation*}
        and
        \begin{equation*}
            \nabla_S f(x^0+d_i;D)^\top D = \delta_f(x^0+d_i;D)^\top D^\dagger D=\delta_f(x^0+d_i;D)^\top,
        \end{equation*}
        for all $i\in\{1,...,p\}$.
        
        Based on these, we also have
        \begin{equation*}
            D^\top\nabla_S^2 f(x^0;D)D= D^\top\left(D^\top\right)^\dagger\delta_{\nabla_S f}(x^0;D)D = \delta_{\nabla_S f}(x^0;D)D = \delta_{\delta_f}(x^0;D),
        \end{equation*}
        where
        \begin{equation}\label{eq:deltaofdeltas}
            \delta_{\delta_f}(x^0;D)=\begin{bmatrix}
                \left(\delta_f(x^0+d_1;D) - \delta_f(x^0;D)\right)^\top\\
                \left(\delta_f(x^0+d_2;D) - \delta_f(x^0;D)\right)^\top\\
                \vdots\\
                \left(\delta_f(x^0+d_p;D) - \delta_f(x^0;D)\right)^\top
            \end{bmatrix}.
        \end{equation}
        Notice that for all $i\in\{1,...,p\}$, the $i$-th diagonal element of $\delta_{\delta_f}(x^0;D)$ is $$f(x^0+2d_i)-2f(x^0+d_i)+f(x^0).$$ Hence, we obtain
        \begin{align*}
            &~~m(x^0+d_i)\\
            &= f(x^0)+\left(2\nabla_S f(x^0;D)-\nabla_S f(x^0;2D)\right)^\top d_i+\frac{1}{2}d_i^\top\nabla_S^2 f(x^0;D)d_i\\
            &= f(x^0)+2f(x^0+d_i)-2f(x^0)-\frac{1}{2}f(x^0+2d_i)+\frac{1}{2}f(x^0) +\frac{1}{2}\left(f(x^0+2d_i)-2f(x^0+d_i)+f(x^0)\right)\\
            &= f(x^0+d_i). \numberthis\label{eq:fullspaceunderdcase2}
        \end{align*}

        {\bf Case \rom{3}:} Suppose $x=x^0+d_i+d_j$ where $i,j\in\{1,...,p\}$.  Notice that for all $i,j\in\{1,...,p\}$, the $ij$-th element of $\delta_{\delta_f}(x^0;D)$ is $$f(x^0+d_i+d_j)-f(x^0+d_i)-f(x^0+d_j)+f(x^0).$$  Using \eqref{eq:fullspaceunderdcase2}, and noticing that $\nabla_S^2 f(x^0;D)$ is symmetric (\cite[Proposition 4.4]{hare2023matrix}), we have
        \begin{align*}
            &~~m(x^0+d_i+d_j)\\
            &= f(x^0)+\left(2\nabla_S f(x^0;D)-\nabla_S f(x^0;2D)\right)^\top \left(d_i+d_j\right) +\frac{1}{2}\left(d_i+d_j\right)^\top \nabla_S^2 f(x^0;D)\left(d_i+d_j\right)\\
            &= m(x^0+d_i) + m(x^0+d_j) + d_i^\top \nabla_S^2 f(x^0;D)d_j - f(x^0)\\
            &= f(x^0+d_i) + f(x^0+d_j) + f(x^0+d_i+d_j)-f(x^0+d_i)-f(x^0+d_j)\\
            &= f(x^0+d_i+d_j).
        \end{align*}

        $\hfill\qed$
    \end{proof}

    Notice that all the sample points needed to construct the model $m(x)$ are contained in the affine space $\mathcal{Y}=\{x^0+D\widehat{s}:\widehat{s}\in\mathbb{R}^p\}$.  Hence, we can construct a model that interpolates $f$ on the image of the points in $\mathbb{R}^p$.  Recall that the $QR$-factorization of a matrix with full-column rank gives an orthonormal basis for its range, determined by $Q$, and the coordinates of its columns under this basis, determined by $R$.  A determined quadratic model can be constructed in the subspace determined by $R$.
    \begin{defini}
        Given $D\in\mathbb{R}^{n\times p}$ with full-column rank, $f:\mathbb{R}^n\to\mathbb{R}$, and $x^0\in\mathbb{R}^n$.  Suppose the $QR$-factorization of $D$ is $D=QR$, where $Q\in\mathbb{R}^{n\times p}$ and $R=[r_1\cdots r_p]\in\mathbb{R}^{p\times p}$.  Let $\widehat{f}(\widehat{s})=f(x^0+Q\widehat{s})$.  The model $\widehat{m}:\mathbb{R}^p\to\mathbb{R}$ is defined by 
        \begin{equation*}
            \widehat{m}(\widehat{s})=\widehat{f}(\mathbf{0})+\left(2\nabla_S \widehat{f}(\mathbf{0};R)-\nabla_S \widehat{f}(\mathbf{0};2R)\right)^\top\widehat{s} + \frac{1}{2}\widehat{s}^\top\nabla_S^2 \widehat{f}(\mathbf{0};R)\widehat{s}.
        \end{equation*}
    \end{defini}
    \begin{theo}\label{thm:subspaced}
        The model $\widehat{m}(\widehat{s})$ is a determined quadratic interpolation model of $\widehat{f}$ in $\mathbb{R}^p$.  In particular, $\widehat{m}(\widehat{s})$ interpolates $\widehat{f}(\widehat{s})$ on all sample points in 
        \begin{equation*}
            \widehat{Y} = \left\{\mathbf{0}\right\}\cup\left\{r_i:i=1,...,p\right\}\cup\left\{r_i+r_j:i,j=1,...,p\right\}.
        \end{equation*}
    \end{theo}
    \begin{proof}
        Since $D$ has full-column rank, we have that $R$ is invertible and the $(p+1)(p+2)/2$ points in the set $\widehat{Y}$ are distinct.  Therefore, we only need to show that $\widehat{m}(\widehat{s})=\widehat{f}(\widehat{s})$ for all $\widehat{s}$ in the set $\widehat{Y}$.  The rest of the proof is analogous to the proof of Theorem \ref{thm:fullspaceunderd}.

        $\hfill\qed$
    \end{proof}

    \begin{rema}
        We note that the term $$2\nabla_S \widehat{f}(\mathbf{0};R)-\nabla_S \widehat{f}(\mathbf{0};2R)$$ in $\widehat{m}(\widehat{s})$ is the adapted centred simplex gradient \cite{chen2023adapting} of $\widehat{f}$ over $\mathbb{Y}=\mathbb{Y}^+\cup\widetilde{\mathbb{Y}}$, where $\mathbb{Y}^+=\{\mathbf{0}\}\cup\{r_i:i=1,...,p\}$ and $\widetilde{\mathbb{Y}}=\{\mathbf{0}\}\cup\{2r_i:i=1,...,p\}$.  Detailed explanations of the notation and more properties of the adapted centred simplex gradient can be found in \cite{chen2023adapting}.
    \end{rema}

    As we noted at the beginning of this subsection, the underdetermined model in $\mathbb{R}^n$ and the determined model in $\mathbb{R}^p$ are equivalent.  The following theorem shows how to convert between them. 
    \begin{theo}\label{thm:modelsrelation}
        Suppose matrix $D=[d_1\cdots d_p]\in\mathbb{R}^{n\times p}$ has full-column rank and let $x^0\in\mathbb{R}^n$.  Suppose the $QR$-factorization of $D$ is $D=QR$, where $Q\in\mathbb{R}^{n\times p}$ and $R=[r_1\cdots r_p]\in\mathbb{R}^{p\times p}$.  Let $\widehat{f}(\widehat{s})=f(x^0+Q\widehat{s})$.  Then $m(x)$ and $\widehat{m}(\widehat{s})$ have the following relation 
        \begin{equation}\label{eq:relation1}
            \widehat{m}(\widehat{s})=m(x^0+Q\widehat{s}).
        \end{equation}
        Moreover, for all $x\in\mathcal{Y}=\{x^0+D\widehat{s}:\widehat{s}\in\mathbb{R}^p\}$, we have 
        \begin{equation}\label{eq:relation2}
            m(x)=\widehat{m}(Q^\top(x-x^0)).
        \end{equation}
    \end{theo}
    \begin{proof}
        Notice that $d_i=Qr_i$ for all $i\in\{1,...,p\}$, therefore we have 
        \begin{align*}
            f(x^0+d_i)-f(x^0)&=\widehat{f}(r_i)-\widehat{f}(\mathbf{0}),\\
            f(x^0+2d_i)-f(x^0)&=\widehat{f}(2r_i)-\widehat{f}(\mathbf{0}),\\
            f(x^0+d_i+d_j)-f(x^0+d_i)&=\widehat{f}(r_i+r_j)-\widehat{f}(r_i),
        \end{align*}
        for all $i,j\in\{1,...,p\}$.  These imply that
        \begin{align*}
            \delta_f(x^0;D) =\delta_{\widehat{f}}(\mathbf{0};R),~~~\delta_f(x^0;2D) =\delta_{\widehat{f}}(\mathbf{0};2R),~~~\delta_f(x^0+d_i;D) =\delta_{\widehat{f}}(r_i;R),
        \end{align*}
        for all $i\in\{1,...,p\}$.
        
        Since $D=QR$ has full-column rank, we have $$D^\dagger=(D^\top D)^{-1}D^\top=(R^\top Q^\top QR)^{-1}R^\top Q^\top=(R^\top R)^{-1}R^\top Q^\top=R^{-1}Q^\top,$$ which gives $D^\dagger Q=R^{-1}$.  Hence, we have
        \begin{equation*}
            \nabla_S f(x^0;D)^\top Q = \delta_f(x^0;D)^\top D^\dagger Q = \delta_{\widehat{f}}(\mathbf{0};R)^\top R^{-1} = \nabla_S \widehat{f}(\mathbf{0};R)^\top,
        \end{equation*}
        \begin{equation*}
            \nabla_S f(x^0;2D)^\top Q= \delta_f(x^0;2D)^\top \left(2D\right)^\dagger Q= \delta_{\widehat{f}}(\mathbf{0};2R)^\top \left(2R\right)^{-1} = \nabla_S \widehat{f}(\mathbf{0};2R)^\top,
        \end{equation*}
        \begin{equation*}
            \nabla_S f(x^0+d_i;D)^\top Q = \delta_f(x^0+d_i;D)^\top D^\dagger Q = \delta_{\widehat{f}}(r_i;R)^\top R^{-1} = \nabla_S \widehat{f}(r_i;R)^\top,
        \end{equation*}
        for all $i\in\{1,...,p\}$.
        
        Based on these, we also have
        \begin{equation*}
            \delta_{\nabla_S f}(x^0;D)Q = 
            \begin{bmatrix}
            \left(\nabla_S \widehat{f}(r_1;R) - \nabla_S \widehat{f}(\mathbf{0};R)\right)^\top\\
            \left(\nabla_S \widehat{f}(r_2;R) - \nabla_S \widehat{f}(\mathbf{0};R)\right)^\top\\
            \vdots\\
            \left(\nabla_S \widehat{f}(r_p;R) - \nabla_S \widehat{f}(\mathbf{0};R)\right)^\top\\
            \end{bmatrix} 
            = \delta_{\nabla_S \widehat{f}}(\mathbf{0};R).
        \end{equation*}
        Therefore,
        \begin{equation*}
            Q^\top\nabla_S^2 f(x^0;D)Q= \left(D^\dagger Q\right)^\top\delta_{\nabla_S f}(x^0;D)Q =\nabla_S^2 \widehat{f}(\mathbf{0};R).
        \end{equation*}
        
        Let $x=x^0+Q\widehat{s}$.  Examining $m(x)$, we obtain
        \begin{align*}
            m(x)&=m(x^0+Q\widehat{s})\\
            &=f(x^0)+\left(2\nabla_S f(x^0;D)-\nabla_S f(x^0;2D)\right)^\top Q\widehat{s} +\frac{1}{2}\widehat{s}^\top Q^\top\nabla_S^2 f(x^0;D)Q\widehat{s}\\
            &=\widehat{f}(\mathbf{0})+\left(2\nabla_S \widehat{f}(\mathbf{0};R)-\nabla_S \widehat{f}(\mathbf{0};2R)\right)^\top\widehat{s} +\frac{1}{2}\widehat{s}^\top \nabla_S^2 \widehat{f}(\mathbf{0};R)\widehat{s}\\
            &=\widehat{m}(\widehat{s}),
        \end{align*}
        which gives the relation in Equation \eqref{eq:relation1}.

        For the second relation, notice that for all $x\in\mathcal{Y}$, there exists $\widehat{s}\in\mathbb{R}^p$ such that $x=x^0+Q\widehat{s}$, or equivalently, $\widehat{s}=Q^\top(x-x^0)$.  The second relation \eqref{eq:relation2} follows from applying this change of notation to relation \eqref{eq:relation1}. 

        $\hfill\qed$
    \end{proof}

    \subsection{Complete algorithm}
    In this subsection, we give the complete $\myalg$ algorithm which consists of the main algorithm and three subroutines.
    
    \begin{algorithm}[H]
    \caption{Main algorithm}
    \DontPrintSemicolon
        \KwIn{\begin{tabular}{lcl}
            $f$ & $f:\mathbb{R}^n\to\mathbb{R}$ & objective function\\
            $p_{\mathrm{rand}}$, $p$ & $1 \le p_{\mathrm{rand}} \le p \le n$ & minimum randomized and full subspace dimension\\
            $x_0$, $\Delta_0$ & $x_0\in\mathbb{R}^n, \Delta_0 > 0$ & starting point and initial trust-region radius\\
            $\Delta_{\min}$, $\Delta_{\max}$ &  $0 < \Delta_{\min} \le \Delta_{\max}$ & minimum and maximum trust-region radii\\
            $\gamma_{\mathrm{dec}}$, $\gamma_{\mathrm{inc}}$ & $0 < \gamma_{\mathrm{dec}} < 1 < \gamma_{\mathrm{inc}}$ & trust-region radius updating parameters\\
            $\eta_1$, $\eta_2$ & $0 < \eta_1 \le \eta_2 < 1$ & acceptance thresholds\\
            $\mu$ & $\mu>0$ & criticality constant\\
            $\epsilon_{\mathrm{rad}}$, $\epsilon_{\mathrm{geo}}$ & $\epsilon_{\mathrm{rad}}\ge 1$, $\epsilon_{\mathrm{geo}}>0$ & sample set radius and geometry tolerance\\
            $M_{A}$ & $M_{A}\ge 1$ & normal distribution matrix upper bound
        \end{tabular}}
        {\bf Initialize:} Generate random mutually orthogonal directions $d_1,..., d_p \in \mathbb{R}^n$ with length $\Delta_0$ by Algorithm \ref{alg:dirngen}; define initial direction matrix $D_0=[d_1\cdots d_p]$.
        
        \For{$k=0,1,...$}
        { 
            Perform the $QR$-factorization $D_k=Q_kR_k$ and construct model $\widehat{m}_k$ at $x_k$.
            
            \uIf{$\mu\left\|\nabla\widehat{m}_k(\mathbf{0})\right\|<\Delta_k$ \tcp*[f]{Criticality test}}
            {
                Set $\Delta_{k+1}=\gamma_{\mathrm{dec}}\Delta_k,x_{k+1}=x_k$ and $D_{k+1}=\gamma_{\mathrm{dec}}D_k$.
            }
            \Else(\tcp*[f]{Trust-region subproblem and update})
            {
                Solve (approximately) $$\min\limits_{\widehat{s}_k\in\mathbb{R}^p}\widehat{m}_k(\widehat{s}_k)~~~~s.t.~\left\|\widehat{s}_k\right\|\le\Delta_k$$ and calculate the step $s_k = Q_k\widehat{s}_k\in\mathbb{R}^n$.
                        
                Calculate $$\rho_k=\frac{f(x_k)-f(x_k+s_k)}{\widehat{m}_k(\mathbf{0})-\widehat{m}_k(\widehat{s}_k)}$$ and update the trust-region radius
                \begin{equation*}
                  \Delta_{k+1}=
                    \begin{cases}
                      \gamma_{\mathrm{dec}}\Delta_k, & \text{if }\rho_k<\eta_1,\\
                      \min(\gamma_{\mathrm{inc}}\Delta_k, \Delta_{\max}), &\text{if } \rho_k>\eta_2  \text{ and } \left\|\widehat{s}_k\right\|\ge 0.95\Delta_k,\\
                      \Delta_k, & \text{otherwise}.
                      % \Delta_k, & \eta_1\le\rho_k\le\eta_2 \text{ or } (\rho_k>\eta_2 \text{ and } \left\|\widehat{s}_k\right\|< 0.95\Delta_k),\\
                    \end{cases}       
                \end{equation*}

                Let $x_{k+1}$ be any point such that
                \begin{equation}\label{eq:nextiterate}
                    f(x_{k+1})\le\min\left\{\{f(x_k+s_k)\}\cup\{f(x_k+d_i+d_j):d_i,d_j\in [\mathbf{0}~D_k]\}\right\}.
                \end{equation}
                
                Create $D^U_{k+1}$ according to Algorithm \ref{alg:mainrm} and denote the columns of $D^U_{k+1}$ by $d^U_1, ..., d^U_{p_1}$ where $0\le p_1\le p$.
                
                Generate $q=p-p_1$ random mutually orthogonal directions $d^R_1,...,d^R_q$ with length $\Delta_{k+1}$ that are orthogonal to all columns in $D^U_{k+1}$ using Algorithm \ref{alg:dirngen}.
                
                Set $D_{k+1}=[d^U_1\cdots d^U_{p_1}~d^R_1\cdots d^R_q]$.
            }
            \lIf{$\Delta_{k+1}<\Delta_{\min}$,}{\textbf{stop}. \tcp*[f]{Stopping test}}
        }
    \end{algorithm}
    
    \begin{algorithm}[H]\caption{Creating $D^U_{k+1}$ algorithm}\label{alg:mainrm}
    \DontPrintSemicolon
        \KwIn{$D_k$; $x_k$; $x_{k+1}$; $s_k$; $\Delta_{k+1}$; $p$; $p_{\mathrm{rand}}$; $\epsilon_{\mathrm{rad}}$; $\epsilon_{\mathrm{geo}}$.}
        Select $p$ directions from $\{\{x_k+s_k-x_{k+1}\}\cup\{x_k+d_i+d_j-x_{k+1}:d_i,d_j\in [\mathbf{0}~D_k]\}\}$ as the columns of $D^U_{k+1}$.
        
        Remove $p_{\mathrm{rand}}$ directions from $D^U_{k+1}$ using Algorithm \ref{alg:ptrm}.
        
        Remove all directions $d\in D^U_{k+1}$ with $\|d\|>\epsilon_{\mathrm{rad}}\Delta_{k+1}$.
        
        \While{$D^U_{k+1}\text{ is not empty and }\sigma_{\min}(D^U_{k+1})<\epsilon_{\mathrm{geo}}$}
        {
            Remove 1 direction from $D^U_{k+1}$ using Algorithm \ref{alg:ptrm}.
        }
    \end{algorithm}
    
    \begin{algorithm}[H]\caption{Direction removing algorithm}\label{alg:ptrm}
    \DontPrintSemicolon
        \KwIn{$D^U_{k+1}$; $\Delta_{k+1}$; number of directions to be removed $p_{\mathrm{rm}}$.}
        Set the number of directions removed {\tt removed} = 0.
        
        \While{${\tt removed}<p_{\mathrm{rm}}$}
        {
            Denote $D^U_{k+1}=[d^U_1\cdots d^U_m]$.
            
            \For{$i=1,...,m$}
            {
                Define matrix $M_i=\left[d^U_1 \cdots d^U_{i-1} ~d^U_{i+1} \cdots d^U_m\right]$ and compute $$\theta_i=\sigma_{\min}(M_i)\cdot\max\left(\frac{\left\|d^U_i\right\|^4}{\Delta^4_{k+1}},1\right).$$
            }
            Remove the direction with the largest $\theta_i$ from $D^U_{k+1}$, set ${\tt removed}={\tt removed}+1$.
        }
    \end{algorithm}
    
    \begin{algorithm}[H]\caption{Direction generating algorithm (modified from \cite[Algorithm 5]{cartis2023scalable})}\label{alg:dirngen}
    \DontPrintSemicolon
        \KwIn{Orthonormal basis for current subspace $Q$ (optional); number of new directions $q$; length of new directions $\Delta$; $M_{A}$.}
        Generate $A\in\mathbb{R}^{n\times q}$ with $A_{ij}\sim\mathcal{N}(0,1\slash q)$ such that $A$ has full-column rank and $\|A\|\le M_{A}$.
        
        If $Q$ is specified, then calculate $\widetilde{A}=A-QQ^\top A$, otherwise set $\widetilde{A}=A$.
        
        Perform the $QR$-factorization $\widetilde{A}=\widetilde{Q}\widetilde{R}$ and denote $\widetilde{Q}=[\widetilde{q}_1\cdots \widetilde{q}_q]$.
        
        Return $\Delta\widetilde{q}_1,...,\Delta\widetilde{q}_q$.
    \end{algorithm}

    We note that the trust-region subproblem is solved approximately by an implementation in~\cite{cartis2023scalable} of the routine $\mathrm{TRSBOX}$ from $\mathrm{BOBYQA}$ \cite{powell2009bobyqa}.  The way we set the next iterate (inequality~\eqref{eq:nextiterate}) gives some flexibility to $\myalg$ by allowing it to use other optimization algorithms (e.g., heuristics) to improve efficiency without breaking convergence.  We also note that, unlike some other model-based trust-region algorithms (e.g., \cite[Algorithm 11.1]{audet2017derivative}), the stopping condition of $\myalg$ does not check if the model gradient is small enough.  This is because, in $\myalg$, a small model gradient implies a small true gradient only if the current subspace has a certain good quality.  As we will explain in Subsection \ref{subsec:subspaceQual}, the subspaces in $\myalg$ have this good quality with a certain probability and therefore model gradients do not always provide sufficient information on true gradients.
    
%    \begin{algorithm}[H]\caption{Updating $x_k,Y_k$ and $\widehat{m}_k$ algorithm}\label{alg:xoptto2d}
%    \DontPrintSemicolon
%        \KwIn{Current iterate $x_k$ and $Y_k$; current QR-factorization $D_k=Q_kR_k$; current model gradient $g_k$ and model Hessian $H_k$.}
%        Denote $x_k^+=x_k+d_l+d_m=\argmin\{f(x_k+d_i+d_j):d_i,d_j\in D_k\cup\{\mathbf{0}\}\}$.\\
%        \uIf{$d_l=\mathbf{0}$ or $d_m=\mathbf{0}$ (i.e., $x_k^+\in Y_k$)}
%        {
%            Let $D_k^+=D_k,g_k^+=g_k$ and $H_k^+=H_k$.
%        }
%        \Else
%        {
%            Let $D_k^+=[d_1-d_m\cdots -d_m\cdots d_p-d_m]$.\\
%            Perform the $QR$-factorization $D_k^+=Q_k^+R_k^+$ and define $M_{t}=Q_k^\top Q_k^+$.\\
%            Let $g_k^+=M_{t}^\top g_k, H_k^+=M_{t}^\top H_k M_{t}$.
%        }
%        Update $x_k=x_k^+,Y_k=\{x_k^+\}\cup (x_k^++D_k^+), g_k=g_k^+$, and $H_k=H_k^+$.
%    \end{algorithm}

    \subsection{Sample set geometry}\label{subsec:samplesetgeometry}
    In this subsection, we examine the sample set management procedure in $\myalg$ (Algorithms~\ref{alg:mainrm} to \ref{alg:dirngen}).  
    
    In each iteration of $\myalg$, we replace at least $p_{\mathrm{rand}}$ directions to construct the next direction matrix.  Our criteria for dropping and adding directions are based on the idea that $\|D_k^\dagger\|$ should be as small as possible.  We will see in Subsection \ref{subsec:modelAcc} (Theorem \ref{thm:fullyquadconsts}) that $\|D_k^\dagger\|$ is directly related to error bounds and minimizing it improves our bounds on model accuracy.
    
    Our management procedure is such that $\|D_k^\dagger\|=1\slash\sigma_{\min}(D_k)$ is minimized, i.e., $\sigma_{\min}(D_k)$ is maximized.  The following theorem shows the relation between $\sigma_{\min}(D_k)$ and the columns of $D_k$.
    \begin{theo}\label{thm:orthisthebest}
        Let $\widetilde{D}=[d_1\cdots d_{q-1}]\in\mathbb{R}^{n\times (q-1)}$ where $2\le q\le n$. Define $D(x):\mathbb{R}^n\to\mathbb{R}^{n\times q}$ by $D(x)=[\widetilde{D}~x]$.  Let $\Delta>0$.  Then for all $x\in\mathbb{R}^n$ with $\|x\|=\Delta$, we have 
        \begin{equation*}
            \sigma_{\min}(D(x))\le\min\left\{\sigma_{\min}(\widetilde{D}), \Delta\right\}.
        \end{equation*}

        In particular, if $x^*\in\mathbb{R}^n$ with $\|x^*\|=\Delta$ satisfies $d_i^\top x^*=0$ for all $i\in\{1,...,q-1\}$, then 
        \begin{equation*}
            \sigma_{\min}(D(x^*))=\min\left\{\sigma_{\min}(\widetilde{D}), \Delta\right\}.
        \end{equation*}
    \end{theo}
    \begin{proof}
        Let $x\in\mathbb{R}^n$ such that $\|x\|=\Delta$.  Since $q\le n$, we have $\sigma_{\min}(D(x))=\min_{\|y\|=1}\|D(x)y\|$ (\cite[Theorem 3.1.2]{horn1991topics}).  Consider $y=[y_1,...,y_q]^\top\in\mathbb{R}^q$.  Since $q\ge 2$, we can let $y_q=0$ and obtain 
        \begin{equation*}
        \min\limits_{\|y\|=1}\left\|D(x)y\right\|\le\min\limits_{\|y\|=1,y_q=0}\left\|D(x)y\right\|=\min\limits_{\|w\|=1}\left\|\widetilde{D}w\right\|=\sigma_{\min}(\widetilde{D}).
        \end{equation*}
        Clearly,
        \begin{equation*}
        \min\limits_{\|y\|=1}\left\|D(x)y\right\|\le\left\|D(x)e_q^{(q)}\right\|=\left\|x\right\|=\Delta.
        \end{equation*}
        Therefore, we have 
        \begin{equation*}
            \sigma_{\min}(D(x))\le\min\left\{\sigma_{\min}(\widetilde{D}), \Delta\right\},
        \end{equation*}
        which gives the first inequality.

        Now we prove the second equality.  Recall that $\sigma_{\min}(M)=\sqrt{\lambda_{\min}(M^\top M)}$ for any matrix $M\in\mathbb{R}^{n\times z}$ with $z\le n$.  Since $d_i^\top x^*=0$ for all $i\in\{1,...,q-1\}$, we have $\widetilde{D}^\top x^*=\mathbf{0}$.  Hence, 
        \begin{equation*}
            D(x^*)^\top D(x^*) = \begin{bmatrix}
                \widetilde{D}^\top \widetilde{D} & \mathbf{0}^\top\\
                \mathbf{0} & \left(x^*\right)^\top x^*
            \end{bmatrix},
        \end{equation*}
        whose eigenvalues are all eigenvalues of $\widetilde{D}^\top \widetilde{D}$ and $\left(x^*\right)^\top x^*=\Delta^2$.  Therefore, we obtain
        \begin{equation*}
            \sigma_{\min}(D(x^*))=\sqrt{\min\left\{\lambda_{\min}(\widetilde{D}^\top \widetilde{D}), \Delta^2\right\}}=\min\left\{\sigma_{\min}(\widetilde{D}), \Delta\right\},
        \end{equation*}
        which gives the second equality.

        $\hfill\qed$
    \end{proof}

    Theorem \ref{thm:orthisthebest} shows that if we want to add a column to a matrix using a vector with magnitude $\Delta$ such that the minimum singular value of the new matrix is maximized, then the best choice is a vector orthogonal to all the current columns.  This gives us the criterion for adding new directions, which is applied in Algorithm~\ref{alg:dirngen}.  
    
    Notice that Theorem \ref{thm:orthisthebest} also says that the largest possible minimum singular value of the new matrix is the minimum of the current minimum singular value and $\Delta$.  This implies that the subspaces constructed using Algorithms \ref{alg:mainrm} to \ref{alg:dirngen} satisfy 
    \begin{equation*}
    \left\|D_k^\dagger\right\|=
        \begin{cases}
          1\slash\Delta_k, & D^U_k~\text{is empty},\\
          \max\left\{1\slash\sigma_{\min}(D^U_k),1\slash\Delta_k\right\}, & D^U_k~\text{is not empty},
        \end{cases}       
    \end{equation*}
    where the first case comes from the fact that when $D^U_k$ is empty, $D_k$ consists of mutually orthogonal columns with length $\Delta_k$.  Since Algorithm \ref{alg:mainrm} guarantees that $\sigma_{\min}(D^U_k)\ge \epsilon_{\mathrm{geo}}$ whenever $D^U_k$ is not empty, we always have
    \begin{equation}\label{ineq:ebDinverse}
        \left\|D_k^\dagger\right\|\le\max\left\{\frac{1}{\epsilon_{\mathrm{geo}}},\frac{1}{\Delta_{\min}}\right\},
    \end{equation}
    Inequality \eqref{ineq:ebDinverse} is an important property and we will return to it later in Subsection \ref{subsec:modelAcc}.
    
    This also implies a criterion for removing directions.  That is, we want to drop the direction such that the remaining matrix has the largest minimum singular value. This criterion is applied in Algorithm \ref{alg:ptrm}.

    We note that Algorithm \ref{alg:ptrm} is different from the procedure used in \cite{cartis2023scalable}, which calculates the vector of $\theta_i$ values once and simultaneously removes directions with the largest $\theta_i$.  The following example shows that these two procedures may produce different results.
    \begin{exampl}\label{ex:ptrmAlgsCmp}
        Suppose $\Delta_{k+1}=1$.  Suppose we want to drop two directions from $$D^U_{k+1}=\begin{bmatrix}
            \frac{1}{2\sqrt{3}} & \frac{1}{\sqrt{3}} & \frac{1}{\sqrt{3}}\\
            0 & \frac{1}{\sqrt{3}} & \frac{1}{\sqrt{3}}\\
            0 & \frac{1}{\sqrt{3}} & \frac{1}{2\sqrt{3}}
        \end{bmatrix}.$$  Then we have
        \begin{equation*}
            M_1=\begin{bmatrix}
                \frac{1}{\sqrt{3}} & \frac{1}{\sqrt{3}}\\
                \frac{1}{\sqrt{3}} & \frac{1}{\sqrt{3}}\\
                \frac{1}{\sqrt{3}} & \frac{1}{2\sqrt{3}}
            \end{bmatrix},~~M_2=\begin{bmatrix}
                \frac{1}{2\sqrt{3}} & \frac{1}{\sqrt{3}}\\
                0 & \frac{1}{\sqrt{3}}\\
                0 & \frac{1}{2\sqrt{3}}
            \end{bmatrix},~~M_3=\begin{bmatrix}
                \frac{1}{2\sqrt{3}} & \frac{1}{\sqrt{3}}\\
                0 & \frac{1}{\sqrt{3}}\\
                0 & \frac{1}{\sqrt{3}}
            \end{bmatrix},
        \end{equation*}
        with minimum singular values
        \begin{equation*}
            \sigma_{\min}(M_1)\approx 0.180,~~\sigma_{\min}(M_2)\approx 0.210,~~\sigma_{\min}(M_3)\approx 0.232.
        \end{equation*}
        Hence, we have
        \begin{align*}
            \theta_1\approx 0.180\cdot\max\left\{\left(\frac{1}{2\sqrt{3}}\right)^4, 1\right\}&=0.180,\\
            \theta_2\approx 0.210\cdot\max\left\{1^4, 1\right\}&=0.210,\\
            \theta_3\approx 0.232\cdot\max\left\{\left(\frac{3}{4}\right)^4, 1\right\}&=0.232.
        \end{align*}

        According to the procedure used in \cite{cartis2023scalable}, we would drop the last two directions $[1\slash \sqrt{3},1\slash \sqrt{3},1\slash \sqrt{3}]^\top, [1\slash \sqrt{3},1\slash \sqrt{3},1\slash 2\sqrt{3}]^\top$.  The minimum singular value of the remaining matrix is $\sigma_{\min}([1\slash 2\sqrt{3},0,0]^\top)\approx 0.289$.

        However, if we use Algorithm \ref{alg:ptrm}, then we drop the last direction $[1\slash \sqrt{3},1\slash \sqrt{3},1\slash 2\sqrt{3}]^\top$ and update
        \begin{equation*}
            M_1=\begin{bmatrix}
                \frac{1}{\sqrt{3}}\\
                \frac{1}{\sqrt{3}}\\
                \frac{1}{\sqrt{3}}
            \end{bmatrix},~~M_2=\begin{bmatrix}
                \frac{1}{2\sqrt{3}}\\
                0\\
                0
            \end{bmatrix},
        \end{equation*}
        which gives
        \begin{equation*}
            \theta_1= 1,~~\theta_2\approx 0.289.
        \end{equation*}
        Next, we drop the first direction $[1\slash 2\sqrt{3},0,0]^\top$ and the minimum singular value of the remaining matrix is $\sigma_{\min}([1\slash \sqrt{3},1\slash \sqrt{3},1\slash \sqrt{3}]^\top)=1$.
    \end{exampl}

    Since $1>0.289$, in Example \ref{ex:ptrmAlgsCmp}, the result given by Algorithm \ref{alg:ptrm} is better than the procedure used in \cite{cartis2023scalable}.  In fact, if $p_{\mathrm{{rm}}}\le 2$, then the result given by Algorithm \ref{alg:ptrm} is always no worse than the procedure used in \cite{cartis2023scalable}.  Indeed, the first direction removed by both procedures is always the same.  For the second direction, Algorithm \ref{alg:ptrm} calculates the $\theta_i$ of every remaining direction, including the second direction removed by the procedure used in \cite{cartis2023scalable}, and removes the direction with the largest value.  However, if $p_{\mathrm{{rm}}} > 2$, then we do not have guarantees on which procedure gives the better result.  In fact, we have randomly generated matrices $D_{k+1}^U\in\mathbb{R}^{4\times 4}$ and let $p_{\mathrm{rm}}=3$.  Numerical results of running each procedure on these matrices show that there are cases in which each of the procedures produces better results.  Future research could further explore this problem and develop an optimal strategy for removing directions.

    \subsection{Model accuracy}\label{subsec:modelAcc}
    In $\myalg$, we require the models to provide a certain level of accuracy in the subspaces.  Hence, following \cite{cartis2023scalable}, we introduce the notion of $Q$-fully linear models and $Q$-fully quadratic models.

    \begin{defini}
        Given $f\in\mathcal{C}^1,x\in\mathbb{R}^n,\bar{\Delta}>0$, and $Q\in\mathbb{R}^{n\times p}$, we say that $\mathcal{M}_{\bar{\Delta}}=\{\widehat{m}_{\Delta}:\mathbb{R}^p\to\mathbb{R}\}_{\Delta\in (0,\bar{\Delta}]}$ is a class of $Q$-fully linear models of $f$ at $x$ parameterized by $\Delta$ if there exists constants $\kappa_{ef}(x)>0$ and $\kappa_{eg}(x)>0$ such that for all $\Delta\in (0,\bar{\Delta}]$ and $\widehat{s}\in\mathbb{R}^p$ with $\|\widehat{s}\|\le\Delta$,
        \begin{align*}
            \left|f(x+Q\widehat{s})-\widehat{m}_{\Delta}(\widehat{s})\right|&\le\kappa_{ef}(x)\Delta^2,\\
            \left\|Q^\top\nabla f(x+Q\widehat{s})-\nabla\widehat{m}_{\Delta}(\widehat{s})\right\|&\le\kappa_{eg}(x)\Delta.
        \end{align*}
    \end{defini}
    \begin{defini}
        Given $f\in\mathcal{C}^2,x\in\mathbb{R}^n,\bar{\Delta}>0$, and $Q\in\mathbb{R}^{n\times p}$, we say that $\mathcal{M}_{\bar{\Delta}}=\{\widehat{m}_{\Delta}:\mathbb{R}^p\to\mathbb{R}\}_{\Delta\in (0,\bar{\Delta}]}$ is a class of $Q$-fully quadratic models of $f$ at $x$ parameterized by $\Delta$ if there exists constants $\kappa_{ef}(x)>0,\kappa_{eg}(x)>0$, and $\kappa_{eh}(x)>0$ such that for all $\Delta\in (0,\bar{\Delta}]$ and $\widehat{s}\in\mathbb{R}^p$ with $\|\widehat{s}\|\le\Delta$,
        \begin{align*}
            \left|f(x+Q\widehat{s})-\widehat{m}_{\Delta}(\widehat{s})\right|&\le\kappa_{ef}(x)\Delta^3,\\
            \left\|Q^\top\nabla f(x+Q\widehat{s})-\nabla\widehat{m}_{\Delta}(\widehat{s})\right\|&\le\kappa_{eg}(x)\Delta^2,\\
            \left\|Q^\top\nabla^2 f(x+Q\widehat{s})Q-\nabla^2\widehat{m}_{\Delta}(\widehat{s})\right\|&\le\kappa_{eh}(x)\Delta.
        \end{align*}
    \end{defini}

    Clearly, if $\mathcal{M}_{\bar{\Delta}}$ is a class of $Q$-fully quadratic models of $f$ at $x$, then $\mathcal{M}_{\bar{\Delta}}$ is also a class of $Q$-fully linear models of $f$ at $x$.

    In this subsection, we develop error bounds for the models and show that each $\widehat{m}_k$ belongs to a class of $Q_k$-fully quadratic models of $f$ at $x_k$, where $Q_kR_k$ is the $QR$-factorization of $D_k$.  Moreover, we show that constants $\kappa_{ef},\kappa_{eg}$, and $\kappa_{eh}$ are independent of $k$.
    
    We note that the convergence of $\myalg$ only requires that each $\widehat{m}_k$ belongs to a class of $Q_k$-fully linear models (see Section \ref{sec:converg} for details).  As this requirement is weaker than $Q_k$-fully quadratic, our convergence analysis holds for $\widehat{m}_k$.

    We begin by introducing two lemmas that are crucial to the proof of the error bounds.  The first lemma is an application of Taylor's theorem.
    \begin{lemm}\label{lem:Taylorremeb}
        Let $x\in\mathbb{R}^n,\Delta>0$, and $f\in\mathcal{C}^{2+}$ on $B(x;\Delta)$ with constant $L_{\nabla^2 f}$. For all $d\in\mathbb{R}^n$ with $\|d\|\le\Delta$,
        \begin{equation*}
            \left|f(x+d)-f(x)-\nabla f(x)^\top d - \frac{1}{2}d^\top\nabla^2 f(x) d\right|\le\frac{1}{6}L_{\nabla^2 f}\left\|d\right\|^3.
        \end{equation*}
    \end{lemm}
    \begin{proof}
        See \cite[Lemma 4.1.14]{dennis1996numerical}.

        $\hfill\qed$
    \end{proof}

    The second lemma is an error bound for the generalized simplex Hessian.  This result is a simple extension of the error bound presented in \cite{hare2023matrix}.
    \begin{lemm}\label{lem:gsheb}
        Let $x^0\in\mathbb{R}^n,\bar{\Delta}>0$, and $f\in\mathcal{C}^{2+}$ on $B(x^0;\bar{\Delta})$ with constant $L_{\nabla^2 f}$.  Suppose $D=[d_1\cdots d_n]\in\mathbb{R}^{n\times n}$ is invertible and $\appdiam(D)=\max_{1\le i\le n}\|d_i\|\le\bar{\Delta}$.  Then for all $x\in B(x^0;\appdiam(D))$ we have
        \begin{equation*}
            \left\|\nabla^2f(x) - \nabla_S^2f(x^0;D)\right\|\le 4nL_{\nabla^2 f}\left\|\widehat{D}^{-\top}\right\|^2\appdiam(D) + L_{\nabla^2 f}\appdiam(D),
        \end{equation*}
        where $\widehat{D}=D\slash\appdiam(D)$.
    \end{lemm}
    \begin{proof}
        Using Lemma \ref{lem:Taylorremeb}, the proof of \cite[Theorem 4.2]{hare2023matrix} holds for  $f\in\mathcal{C}^{2+}$, we see that
        \begin{equation*}
            \left\|\nabla^2f(x^0) - \nabla_S^2f(x^0;D)\right\|\le 4nL_{\nabla^2 f}\left\|\widehat{D}^{-\top}\right\|^2\appdiam(D).
        \end{equation*}
        The proof is complete by noticing that
        \begin{equation*}
            \left\|\nabla^2f(x) - \nabla_S^2f(x^0;D)\right\|\le \left\|\nabla^2f(x^0) - \nabla_S^2f(x^0;D)\right\| + \left\|\nabla^2f(x) - \nabla^2f(x^0)\right\|
        \end{equation*}
        and
        \begin{align*}
            \left\|\nabla^2f(x) - \nabla^2f(x^0)\right\|\le L_{\nabla^2 f}\left\|x-x^0\right\|\le L_{\nabla^2 f}\appdiam(D).
        \end{align*}

        $\hfill\qed$
    \end{proof}

    Now we can start to develop the error bounds for $\widehat{m}$.  The next lemma shows that restricting $f$ to a subspace determined by an orthonormal basis does not increase its Lipschitz constant.
    \begin{lemm}
        Suppose $f\in\mathcal{C}^{2+}$ with Lipschitz constant $L_{\nabla^2 f}$.  Suppose $Q\in\mathbb{R}^{n\times p}$ consists of $p$ orthonormal columns and let $x\in\mathbb{R}^n$.  Let $\widehat{f}(\widehat{s})=f(x+Q\widehat{s})$.  Then $\nabla^2\widehat{f}$ is Lipschitz continuous on $\mathbb{R}^p$ with constant $L_{\nabla^2 f}$.
    \end{lemm}
%     \begin{proof}
%         Since $Q\in\mathbb{R}^{n\times p}$ has $p$ orthonormal columns, we have $\|Q\|=\|Q^\top\|=1$ and $\|Qd\|=\|d\|$ for all $d\in\mathbb{R}^p$.  Let $\widehat{s},d\in\mathbb{R}^p$.  Then
%         \begin{align*}
%             \left\|\nabla^2\widehat{f}(\widehat{s}+d)-\nabla^2\widehat{f}(\widehat{s})\right\| 
%             &= \left\|Q^\top\nabla^2 f(x+Q(\widehat{s}+d))Q-Q^\top\nabla^2 f(x+Q\widehat{s})Q\right\|\\
%             &\le \left\|Q^\top\right\|\left\|\nabla^2 f(x+Q(\widehat{s}+d))- \nabla^2 f(x+Q\widehat{s})\right\|\left\|Q\right\|\\
%             &= \left\|\nabla^2 f(x+Q(\widehat{s}+d))- \nabla^2 f(x+Q\widehat{s})\right\|\\
%             &\le L_{\nabla^2 f}\left\|Qd\right\|\\
%             &= L_{\nabla^2 f}\left\|d\right\|.
%         \end{align*}
%
%        $\hfill\qed$
%     \end{proof}

    Henceforth, we will not differentiate between the Lipschitz constant of $\nabla^2 f$ on $\mathbb{R}^n$ and $\nabla^2\widehat{f}$ on $\mathbb{R}^p$, and denote them both by $L_{\nabla^2 f}$.
    
    The following theorem gives the error bounds for $\widehat{m}$. The bounds are given in terms of matrix $R$, but in Theorem \ref{thm:fullyquadconsts} we shall explain how the error bounds are related to $\|D_k^\dagger\|$, which is controlled by our sample set management procedure (see Subsection~\ref{subsec:samplesetgeometry} for details).  
    
    In the following results, we suppose $R=[r_1\cdots r_p]$ and denote $$\appdiam(R)=\max\limits_{1\le i\le p}\left\|r_i\right\|~~\text{and}~~\widehat{R}=\frac{1}{\appdiam(R)}R.$$  The proof is inspired by \cite[Theorem 3.16]{conn2009introduction}. 
    \begin{theo}\label{thm:modelebs}
        Suppose $f\in\mathcal{C}^{2+}$ with Lipschitz constant $L_{\nabla^2 f}$.  Suppose $D\in\mathbb{R}^{n\times p}$ has full-column rank and let $x^0\in\mathbb{R}^n$.  Suppose the $QR$-factorization of $D$ is $D=QR$, where $Q\in\mathbb{R}^{n\times p}$ and $R\in\mathbb{R}^{p\times p}$.    Let $\widehat{f}(\widehat{s})=f(x^0+Q\widehat{s})$.  Then for all $\widehat{s}\in\mathbb{R}^p$ with $\|\widehat{s}\|\le\appdiam(R)$, we have
        \begin{equation*}
            \left\|Q^\top\nabla^2 f(x^0+Q\widehat{s})Q-\nabla^2\widehat{m}(\widehat{s})\right\|\le L_{\nabla^2 f}\left(4p\left\|\widehat{R}^{-\top}\right\|^2+1\right)\appdiam(R),
        \end{equation*}
        \begin{equation*}
            \left\|Q^\top\nabla f(x^0+Q\widehat{s})-\nabla\widehat{m}(\widehat{s})\right\|\le 3\sqrt{p}L_{\nabla^2 f}\left\|\widehat{R}^{-\top}\right\|\left(2p\left\|\widehat{R}^{-\top}\right\|^2+1\right)\appdiam(R)^2,
        \end{equation*}
        and 
        \begin{equation*}
        \begin{aligned}
            &~~\left|f(x^0+Q\widehat{s})-\widehat{m}(\widehat{s})\right|\\
            &\le\left(\sqrt{p}L_{\nabla^2 f}\left\|\widehat{R}^{-\top}\right\|\left(6p\left\|\widehat{R}^{-\top}\right\|^2+2\sqrt{p}\left\|\widehat{R}^{-\top}\right\|+3\right) + \frac{2}{3}L_{\nabla^2 f}\right)\appdiam(R)^3.
        \end{aligned}
        \end{equation*}
    \end{theo}
    \begin{proof}
        For simplicity, we denote 
        \begin{equation}\label{eq:modelcompactnotation}
            \widehat{m}(\widehat{s}) = c+\widehat{g}^\top\widehat{s}+\frac{1}{2}\widehat{s}^\top\widehat{H}\widehat{s}
        \end{equation}
        and define the error between the model and true values by
        \begin{equation}\label{eq:modelfnc}
            E^f(\widehat{s}) = \widehat{m}(\widehat{s}) - \widehat{f}(\widehat{s}),
        \end{equation}
        \begin{equation}\label{eq:modelgrad}
            E^g(\widehat{s}) = \nabla\widehat{m}(\widehat{s}) - \nabla\widehat{f}(\widehat{s}) = \widehat{g}+\widehat{H}\widehat{s} - \nabla\widehat{f}(\widehat{s}),
        \end{equation}
        \begin{equation}\label{eq:modelHess}
            E^H(\widehat{s}) = \nabla^2\widehat{m}(\widehat{s}) - \nabla^2\widehat{f}(\widehat{s}) = \widehat{H} - \nabla^2\widehat{f}(\widehat{s}).
        \end{equation}

        According to Lemma \ref{lem:gsheb}, we have
        \begin{equation*}
            \left\|E^H(\widehat{s})\right\|=\left\|\widehat{H}-\nabla^2\widehat{f}(\widehat{s})\right\|\le L_{\nabla^2 f}\left(4p\left\|\widehat{R}^{-\top}\right\|^2+1\right)\appdiam(R),
        \end{equation*}
        which is the first error bound.

        Since $\widehat{m}$ interpolates $\widehat{f}$ on $\{\mathbf{0}\}\cup\{r_i:i=1,...,p\}$ (Theorem \ref{thm:subspaced}), we have $$\widehat{m}(r_i)=c+\widehat{g}^\top r_i+\frac{1}{2}r_i^\top\widehat{H}r_i=\widehat{f}(r_i)$$ for all $r_i\in[\mathbf{0}~r_1\cdots r_p]$.  Subtracting \eqref{eq:modelcompactnotation} from this equation, and noticing that $\widehat{H}$ is symmetric (\cite[Proposition 4.4]{hare2023matrix}), we obtain
        \begin{equation}\label{eq:mr-ms}
            \widehat{m}(r_i)-\widehat{m}(\widehat{s}) = 
            \left(r_i-\widehat{s}\right)^\top\widehat{g} + \frac{1}{2}\left(r_i-\widehat{s}\right)^\top\widehat{H}\left(r_i-\widehat{s}\right) + \left(r_i-\widehat{s}\right)^\top\widehat{H}\widehat{s} = \widehat{f}(r_i)-\widehat{m}(\widehat{s})
        \end{equation}
        for all $r_i\in[\mathbf{0}~r_1\cdots r_p]$.  Substituting the expression of $\widehat{m}(\widehat{s})$, $\widehat{g}$, and $\widehat{H}$ from \eqref{eq:modelfnc} to \eqref{eq:modelHess} into \eqref{eq:mr-ms}, then regrouping terms, we get
        \begin{equation}\label{eq:CSVeq3.16}
        \begin{aligned}
            \left(r_i-\widehat{s}\right)^\top E^g(\widehat{s}) + \frac{1}{2}\left(r_i-\widehat{s}\right)^\top E^H(\widehat{s})\left(r_i-\widehat{s}\right) &= \widehat{f}(r_i)-\widehat{f}(\widehat{s})-\left(r_i-\widehat{s}\right)^\top\nabla\widehat{f}(\widehat{s})\\
            &~~~-\frac{1}{2}\left(r_i-\widehat{s}\right)^\top\nabla^2\widehat{f}(\widehat{s})\left(r_i-\widehat{s}\right)-E^f(\widehat{s}),
        \end{aligned}
        \end{equation}
        for all $r_i\in[\mathbf{0}~r_1\cdots r_p]$.  In particular, when $r_i=\mathbf{0}$, we have
        \begin{equation}\label{eq:CSVeq3.16at0}
            -\widehat{s}^\top E^g(\widehat{s}) + \frac{1}{2}\widehat{s}^\top E^H(\widehat{s})\widehat{s} = \widehat{f}(\mathbf{0})-\widehat{f}(\widehat{s})-\left(-\widehat{s}\right)^\top\nabla\widehat{f}(\widehat{s})-\frac{1}{2}\left(-\widehat{s}\right)^\top\nabla^2\widehat{f}(\widehat{s})\left(-\widehat{s}\right)-E^f(\widehat{s}).
        \end{equation}
        Since $\widehat{H}$ and $\nabla^2\widehat{f}(\widehat{s})$ are symmetric, we have $E^H(\widehat{s})$ is symmetric.  Subtracting \eqref{eq:CSVeq3.16at0} from \eqref{eq:CSVeq3.16}, then regrouping terms, we get  
        \begin{equation}\label{eq:CSVeq3.16aftercanceling}
            r_i^\top E^g(\widehat{s}) = r_i^\top E^H(\widehat{s})\widehat{s}  - \frac{1}{2}r_i^\top E^H(\widehat{s})r_i + \mathrm{Rem}_i - \mathrm{Rem}_0
        \end{equation}
        for all $r_i\in[r_1\cdots r_p]$, where
        \begin{equation*}
            \mathrm{Rem}_i = \widehat{f}(r_i)-\widehat{f}(\widehat{s})-\left(r_i-\widehat{s}\right)^\top\nabla\widehat{f}(\widehat{s})-\frac{1}{2}\left(r_i-\widehat{s}\right)^\top\nabla^2\widehat{f}(\widehat{s})\left(r_i-\widehat{s}\right)
        \end{equation*}
        and 
        \begin{equation*}
            \mathrm{Rem}_0 = \widehat{f}(\mathbf{0})-\widehat{f}(\widehat{s})-\left(-\widehat{s}\right)^\top\nabla\widehat{f}(\widehat{s})-\frac{1}{2}\left(-\widehat{s}\right)^\top\nabla^2\widehat{f}(\widehat{s})\left(-\widehat{s}\right).
        \end{equation*}
        According to Lemma \ref{lem:Taylorremeb}, the above two terms have upper bounds
        \begin{equation*}
            \left|\mathrm{Rem}_i\right|\le\frac{1}{6}L_{\nabla^2 f}\left\|r_i-\widehat{s}\right\|^3~~~\text{and}~~~ \left|\mathrm{Rem}_0\right|\le\frac{1}{6}L_{\nabla^2 f}\left\|-\widehat{s}\right\|^3.
        \end{equation*}
        Notice that $\|r_i-\widehat{s}\|\le 2\appdiam(R)$, so \eqref{eq:CSVeq3.16aftercanceling} gives
        \begin{align*}
            \left|r_i^\top E^g(\widehat{s})\right| &= \left|r_i^\top E^H(\widehat{s})\widehat{s} - \frac{1}{2}r_i^\top E^H(\widehat{s})r_i + \mathrm{Rem}_i - \mathrm{Rem}_0\right|\\
            &\le \left\|E^H(\widehat{s})\right\|\left\|\widehat{s}\right\|\left\|r_i\right\| + \frac{1}{2}\left\|E^H(\widehat{s})\right\|\left\|r_i\right\|^2 + \left|\mathrm{Rem}_i\right| +\left| \mathrm{Rem}_0\right|\\
            &\le \left(\left\|E^H(\widehat{s})\right\| + \frac{1}{2}\left\|E^H(\widehat{s})\right\|\right)\appdiam(R)^2 + \frac{1}{6}L_{\nabla^2 f}\left\|r_i-\widehat{s}\right\|^3 + \frac{1}{6}L_{\nabla^2 f}\left\|-\widehat{s}\right\|^3\\
            &\le L_{\nabla^2 f}\left(6p\left\|\widehat{R}^{-\top}\right\|^2 + \frac{3}{2}\right)\appdiam(R)^3 + \frac{1}{6}L_{\nabla^2 f}\left(8\appdiam(R)^3 +\appdiam(R)^3\right)\\
            &= 3L_{\nabla^2 f}\left(2p\left\|\widehat{R}^{-\top}\right\|^2+1\right)\appdiam(R)^3,
        \end{align*}
        for all $r_i\in[r_1\cdots r_p]$.  Therefore, we obtain the second error bound
        \begin{align*}
            \left\|Q^\top\nabla f(x^0+Q\widehat{s})-\nabla\widehat{m}(\widehat{s})\right\| &=\left\|E^g(\widehat{s})\right\|\\
            &= \left\|\widehat{R}^{-\top}\widehat{R}^\top E^g(\widehat{s})\right\|\\
            &\le \left\|\widehat{R}^{-\top}\right\|\left\|\widehat{R}^\top E^g(\widehat{s})\right\|\\
            &= \left\|\widehat{R}^{-\top}\right\|\frac{1}{\appdiam(R)}\sqrt{\sum\limits_{i=1}^p\left|r_i^\top E^g(\widehat{s})\right|^2}\\
            &\le 3\sqrt{p}L_{\nabla^2 f}\left\|\widehat{R}^{-\top}\right\|\left(2p\left\|\widehat{R}^{-\top}\right\|^2+1\right)\appdiam(R)^2.
        \end{align*}\
    
        For the third error bound, we regroup \eqref{eq:CSVeq3.16at0} to get
        \begin{align*}
            E^f(\widehat{s}) = \mathrm{Rem}_0 + \widehat{s}^\top E^g(\widehat{s}) - \frac{1}{2}\widehat{s}^\top E^H(\widehat{s})\widehat{s}.
        \end{align*}
        Applying the error bounds for $\left\|E^g(\widehat{s})\right\|$ and $\left\|E^H(\widehat{s})\right\|$, we obtain
        \begin{align*}
            \left|E^f(\widehat{s})\right| &\le \left|\mathrm{Rem}_0\right| + \left\|E^g(\widehat{s})\right\|\left\|\widehat{s}\right\| + \frac{1}{2}\left\|E^H(\widehat{s})\right\|\left\|\widehat{s}\right\|^2\\
            &= \left(\sqrt{p}L_{\nabla^2 f}\left\|\widehat{R}^{-\top}\right\|\left(6p\left\|\widehat{R}^{-\top}\right\|^2+2\sqrt{p}\left\|\widehat{R}^{-\top}\right\|+3\right) + \frac{2}{3}L_{\nabla^2 f}\right)\appdiam(R)^3.
        \end{align*}

        $\hfill\qed$
    \end{proof}

    Now we show that each $\widehat{m}_k$ belongs to a class of $Q_k$-fully quadratic models of $f$ at $x_k$ parameterized by $\Delta_k\in(0,\Delta_{\max}]$ and the constants $\kappa_{ef},\kappa_{eg}$, and $\kappa_{eh}$ are independent of $k$.   To do this, we need to find the relation between $\appdiam(R_k)$ and $\Delta_k$, and show that the coefficients of $\Delta_k$ and $\Delta_k^2$ in the new error bounds do not depend on $k$.
    
    We first note that $D_k=[d_1\cdots d_p]=Q_kR_k=[Q_kr_1\cdots Q_kr_p]$, so $$\appdiam(R_k)=\max\limits_{1\le i\le p}\left\|r_i\right\|=\max\limits_{1\le i\le p}\left\|Q_kr_i\right\|=\max\limits_{1\le i\le p}\left\|d_i\right\|.$$
    Notice that $D_k$ is constructed in the $(k-1)$-st iteration using Algorithms \ref{alg:mainrm} to \ref{alg:dirngen}.  Hence, any directions with a length larger than $\epsilon_{\mathrm{rad}}\Delta_{k}$ are removed and all new directions have a length equal to $\Delta_k$.  Therefore, we have $\appdiam(R_k)\le\epsilon_{\mathrm{rad}}\Delta_{k}\le\epsilon_{\mathrm{rad}}\Delta_{\max}$.  Since each direction matrix $D_k$ in $\myalg$ contains at least one new direction created by Algorithm \ref{alg:dirngen}, we also have $\Delta_k\le\appdiam(R_k)$.
    
    Finally to prove the following theorem, we also need the result of Subsection \ref{subsec:samplesetgeometry} that $\|D_k^\dagger\|\le\max\{1\slash\epsilon_{\mathrm{geo}}, 1\slash\Delta_{\min}\}$ (inequality~\eqref{ineq:ebDinverse}).
    
    \begin{theo}\label{thm:fullyquadconsts}
        Suppose $f\in\mathcal{C}^{2+}$ with Lipschitz constant $L_{\nabla^2 f}$.  Suppose $D_k\in\mathbb{R}^{n\times p}$ has full-column rank and let $x_k\in\mathbb{R}^n$.  Suppose the $QR$-factorization of $D_k$ is $D_k=Q_kR_k$, where $Q_k\in\mathbb{R}^{n\times p}$ and $R_k\in\mathbb{R}^{p\times p}$.  Let $\widehat{f}(\widehat{s})=f(x_k+Q_k\widehat{s})$.  Then $\widehat{m}_k$ belongs to a class of $Q_k$-fully quadratic models of $f$ at $x_k$ parameterized by $\Delta_k\in(0,\Delta_{\max}]$ with the following constants (which do not depend on~$k$)
        \begin{align*}
            \kappa_{ef} &= \left(\sqrt{p}L_{\nabla^2 f}M_{\widehat{R}^{-\top}}\left(6pM_{\widehat{R}^{-\top}}^2+2\sqrt{p}M_{\widehat{R}^{-\top}}+3\right) + \frac{2}{3}L_{\nabla^2 f}\right)\epsilon_{\mathrm{rad}}^3,\\
            \kappa_{eg} &= 3\sqrt{p}L_{\nabla^2 f}M_{\widehat{R}^{-\top}}\left(2pM_{\widehat{R}^{-\top}}^2+1\right)\epsilon_{\mathrm{rad}}^2,\\
            \kappa_{eh} &= L_{\nabla^2 f}\left(4pM_{\widehat{R}^{-\top}}^2+1\right)\epsilon_{\mathrm{rad}},
        \end{align*}
        where $M_{\widehat{R}^{-\top}}=\max\{1\slash\epsilon_{\mathrm{geo}}, 1\slash\Delta_{\min}\}\epsilon_{\mathrm{rad}}\Delta_{\max}$.
    \end{theo}
    \begin{proof}
        Since $D_k=Q_kR_k$ has full-column rank, we have $D_k^\dagger=R_k^{-1}Q_k^\top$, that is, $R_k^{-1}=D_k^\dagger Q_k$.  Therefore,
        \begin{equation*}
            \left\|\widehat{R}_k^{-\top}\right\| = \left\|R_k^{-1}\right\|\appdiam(R_k) = \left\|D_k^\dagger Q_k\right\|\appdiam(R_k) \le \left\|D_k^\dagger\right\|\appdiam(R_k) \le M_{\widehat{R}^{-\top}}.
        \end{equation*}
        
        Notice that $\Delta_k\le\appdiam(R_k)\le\epsilon_{\mathrm{rad}}\Delta_k\le\epsilon_{\mathrm{rad}}\Delta_{\max}$. Applying all the above inequalities to Theorem~\ref{thm:modelebs}, we see that $\widehat{m}_k(\widehat{s})$ belongs to a class of $Q_k$-fully quadratic models of $f$ at $x_k$ parameterized by $\Delta_k\in(0,\Delta_{\max}]$ with constants $\kappa_{ef},\kappa_{eg}$, and $\kappa_{eh}$.

        $\hfill\qed$
    \end{proof}

    Theorem \ref{thm:fullyquadconsts} implies that $\widehat{m}_k$ also belongs to a class of $Q_k$-fully linear models of $f$ at $x_k$ parameterized by $\Delta_k\in(0,\Delta_{\max}]$ with constants $\kappa_{ef}$ and $\kappa_{eg}$.  As we explained at the beginning of this subsection, we will only use this fact in the convergence analysis of $\myalg$ in Section~\ref{sec:converg}.

    \subsection{Subspace quality}\label{subsec:subspaceQual}
    In order to prove the convergence of our algorithm, in addition to model accuracy, we also need our subspaces to have a certain good quality, which we now describe.  Inspired by \cite{cartis2023scalable}, we introduce the notion of $\alpha$-well aligned matrices.
    
    \begin{defini}
        Given $f\in\mathcal{C}^1,x\in\mathbb{R}^n$, and $\alpha\in (0,1)$, we say that $A\in\mathbb{R}^{n\times z}$ is $\alpha$-well aligned for $f$ at $x$ if $$\left\|A^\top \nabla f(x)\right\|\ge\alpha\left\|\nabla f(x)\right\|.$$
    \end{defini}

    In this subsection, we show that there exists $\alpha_D \in (0,1)$ independent of $k$ such that with nonzero probability each $D_k\in\mathbb{R}^{n\times p}$ constructed in $\myalg$ is $\alpha_D$-well aligned for $f$ at $x_k$. We begin by recalling one method to generate $\alpha$-well aligned matrices with a certain probability from~\cite{dzahini2022stochastic}.
    \begin{theo}\label{thm:constructwellaligned}
        Suppose $\alpha,\delta_S\in (0,1)$ and $z\ge 4(1-\alpha)^{-2}\ln(1\slash\delta_S)$.  Let $A\in\mathbb{R}^{n\times z}$ be a random matrix such that $A_{ij}\sim\mathcal{N}(0,1\slash z)$.  Then for any deterministic vector $v\in\mathbb{R}^n$, $$\mathbb{P}\left[\left\|A^\top v\right\|\ge\alpha\left\|v\right\|\right]\ge 1-\delta_S.$$
        In particular, given $f\in\mathcal{C}^1$ and $x\in\mathbb{R}^n$, $A$ is $\alpha$-well aligned for $f$ at $x$ with probability at least $1-\delta_S$, i.e., $$\mathbb{P}\left[\left\|A^\top \nabla f(x)\right\|\ge\alpha\left\|\nabla f(x)\right\|\right]\ge 1-\delta_S.$$
    \end{theo}
    \begin{proof}
        See \cite[Theorem 3.1, Corollary 3.1]{dzahini2022stochastic}.

        $\hfill\qed$
    \end{proof}

    In order to save function evaluations, in each iteration of $\myalg$, only part of the matrix $D_k$ is updated.  As we explained in Subsection \ref{subsec:samplesetgeometry}, this is done in order to maintain good geometry of the sample set.  The updated matrix, generated according to Algorithm \ref{alg:dirngen}, comes from the $QR$-factorization of a random matrix $A$ having the form described in Theorem \ref{thm:constructwellaligned}.  Thus we now explain how $QR$-factorization influences $\alpha$-well alignedness.  
    
    The next lemma shows that, if matrix $A$ is $\alpha$-well aligned for $f$ at $x$ and has full-column rank, then the $Q$ matrix from the $QR$-factorization is $(\alpha\slash\|A\|)$-well aligned for $f$ at $x$.
    \begin{lemm}\label{lem:wellalignedQR}
        Suppose $f\in\mathcal{C}^1$ and let $x\in\mathbb{R}^n$.  Let $\alpha\in(0,1)$ and suppose that $A\in\mathbb{R}^{n\times z}$ is $\alpha$-well aligned for $f$ at $x$ and has full-column rank.  Let the $QR$-factorization of $A$ be $A=QR$, where $Q\in\mathbb{R}^{n\times z}$ and $R\in\mathbb{R}^{z\times z}$.  Then $Q$ is $(\alpha\slash\|A\|)$-well aligned for $f$ at $x$.  
        %In particular, we have
        %\begin{equation*}
        %    \mathbb{P}\left[\left\|Q^\top v\right\| \ge \frac{\alpha}{\left\|A\right\|}\left\|v\right\|\right] \ge 1-\delta_S.
        %\end{equation*}
    \end{lemm}
    \begin{proof}
        Since $A$ has full-column rank, we have $\|R\|=\|QR\|=\|A\|\neq 0$.  Hence,  
        \begin{equation*}
            \alpha\left\|\nabla f(x)\right\| \le \left\|A^\top \nabla f(x)\right\| \le \left\|R^\top\right\|\left\|Q^\top \nabla f(x)\right\| = \left\|A\right\|\left\|Q^\top \nabla f(x)\right\|,
        \end{equation*}
        which implies
        \begin{align*}
            \left\|Q^\top \nabla f(x)\right\| \ge \frac{\alpha}{\left\|A\right\|}\left\|\nabla f(x)\right\|.
        \end{align*}

        $\hfill\qed$
    \end{proof}
    
    Now we prove that there exists a constant $\alpha_D\in(0,1)$ independent of $k$ such that each $D_k\in\mathbb{R}^{n\times p}$ constructed in $\myalg$ is $\alpha_D$-well aligned for $f$ at $x_k$ with probability at least $1-\delta_S$.
    
    We denote the randomly generated part of $D_k$ by $D_k^R\in\mathbb{R}^{n\times q}$, where $p_{\mathrm{rand}}\le q\le p$.  This part is constructed by Algorithm \ref{alg:dirngen}.  The unchanged part of $D_k$ is denoted by $D_k^U\in\mathbb{R}^{n\times (p-q)}$.  Hence, the matrix $D_k$ has the form $D_k=[D_k^U~D_k^R]$.

    We require the following lemma, which gives a uniform bound on the norm of $D_k^U$.
    \begin{lemm}\label{lem:uniformlybdDU}
        Let $M_{D^U}=\epsilon_{\mathrm{rad}}\Delta_{\max}\sqrt{p}$.  Then $D^U_k$ satisfies $\|D^U_k\|\le M_{D^U}$ for all $k$.
    \end{lemm}
    \begin{proof}
        Notice that $D_k^U=[d_1^U\cdots d_{p-q}^U]\in\mathbb{R}^{n\times (p-q)}$ is constructed by Algorithm \ref{alg:mainrm} in the $(k-1)$-st iteration.  We have $\|d_i^U\|\le\epsilon_{\mathrm{rad}}\Delta_k$ for all $i=1,...,p-q$.  Consider $y=[y_1,...,y_{p-q}]^\top\in\mathbb{R}^{p-q}$.  Using the AM-QM inequality, we get
        \begin{align*}
            \left\|D_k^U\right\|=\max\limits_{\|y\|=1}\left\|\sum\limits_{i=1}^{p-q}y_id_i^U\right\|\le\max\limits_{1\le i\le p-q}\left\|d_i^U\right\|\max\limits_{\|y\|=1}\sum\limits_{i=1}^{p-q}\left|y_i\right|\le\epsilon_{\mathrm{rad}}\Delta_k\sqrt{p-q}\le M_{D^U}.
        \end{align*}

        $\hfill\qed$
    \end{proof}
    
    \begin{theo}\label{thm:alphaDwellalignedD}
        Suppose $f\in\mathcal{C}^1$.  Suppose $\alpha,\delta_S\in (0,1)$ and $q\ge 4(1-\alpha)^{-2}\ln(1\slash\delta_S)$.  Suppose $D_k^R$ is constructed by Algorithm \ref{alg:dirngen} with $\Delta=\Delta_k$ and $M_A\ge 1$.  Let $$\alpha_D=\min\left\{\frac{\epsilon_{\mathrm{geo}}^2}{M_{D^U}},\frac{\alpha\Delta_{\min}}{2M_{A}}\right\},$$ where $M_{D^U}=\epsilon_{\mathrm{rad}}\Delta_{\max}\sqrt{p}$.  Then $D_k=[D_k^U~D_k^R]$ is $\alpha_D$-well aligned for $f$ at $x_k$ with probability at least $1-\delta_S$.  
    \end{theo}
    \begin{proof}
        Let $A$ be the matrix in Algorithm \ref{alg:dirngen} with $A_{ij}\sim\mathcal{N}(0,1\slash q)$, full-column rank, and $\|A\|\le M_{A}$.  For simplicity, we denote $g=\nabla f(x_k)$.
        
        {\bf Case \rom{1}:} Suppose that in Algorithm \ref{alg:dirngen} we set $\widetilde{A}=A$.  
        
        According to Theorem \ref{thm:constructwellaligned}, $A$ is $\alpha$-well aligned for $f$ at $x_k$ with probability at least $1-\delta_S$.  From Lemma \ref{lem:wellalignedQR}, we have $\widetilde{Q}$ is $(\alpha\slash\|A\|)$-well aligned for $f$ at $x_k$.  Notice that $D^R_k=\Delta_k\widetilde{Q}$ and $\|A\|\le M_{A}$.  Thus we obtain
        \begin{equation*}
            \left\|D_k^\top g\right\| \ge \left\|\left(D_k^R\right)^\top g\right\| = \Delta_k\left\|\widetilde{Q}^\top g\right\| \ge \frac{\alpha\Delta_k}{\|A\|}\left\|g\right\| \ge \frac{\alpha\Delta_{\min}}{M_{A}}\left\|g\right\|\ge\alpha_D\left\|g\right\|.
        \end{equation*}

        {\bf Case \rom{2}:} Suppose that in Algorithm \ref{alg:dirngen} we set $\widetilde{A}=A-QQ^\top A$, where $Q$ is the orthonormal basis for the subspace determined by $D^U_k$.

        Decompose $g$ into $g=g_Q + g_{Q^\perp}$, where $g_Q=QQ^\top g$ is the projection of $g$ onto the column space of $Q$ and $g_{Q^\perp}=(I_n-QQ^\top)g$ is the projection of $g$ onto the orthogonal complement of the column space of $Q$.  Hence, we have $Q^\top g = Q^\top g_Q$ and $(I_n-QQ^\top)^\top g=(I_n-QQ^\top)g=g_{Q^\perp}$. 

        Applying Theorem \ref{thm:constructwellaligned} with $v=g_{Q^\perp}$, we have $\|A^\top g_{Q^\perp}\|\ge\alpha\|g_{Q^\perp}\|$ with probability at least $1-\delta_S$.  Notice that $D^R_k=\Delta_k\widetilde{Q}$ and $\widetilde{A}=(I_n-QQ^\top)A=\widetilde{Q}\widetilde{R}$.  Thus we have $D^R_k=\Delta_k(I_n-QQ^\top)A\widetilde{R}^{-1}$ which gives
        \begin{align}\label{eq:DkTgexpression}
        \begin{split}
            \left\|D_k^\top g\right\| &= \sqrt{\left\|\left(D_k^U\right)^\top g\right\|^2 + \left\|\left(D_k^R\right)^\top g\right\|^2}\\
            &= \sqrt{\left\|\left(D_k^U\right)^\top g_Q\right\|^2 + \Delta_k^2\left\|\widetilde{R}^{-\top}A^\top(I_n-QQ^\top) g\right\|^2}\\
            &= \sqrt{\left\|\left(D^U_k\right)^\top g_Q\right\|^2 + \Delta_k^2\left\|\widetilde{R}^{-\top}A^\top g_{Q^\perp}\right\|^2}.
        \end{split}
        \end{align}

        Now we find lower bounds for each term under the square root symbol.  Since the column space of $Q$ is the same as the column space of $D^U_k$, there exists a vector $w\in\mathbb{R}^p$ such that $g_Q=D^U_k w$.  Hence,
        \begin{equation*}
            \left\|\left(D^U_k\right)^\top g_Q\right\| \ge \sigma_{\min}\left(\left(D^U_k\right)^\top D^U_k\right)\left\|w\right\| = \sigma_{\min}\left(D^U_k\right)^2\left\|w\right\| \ge \epsilon_{\mathrm{geo}}^2\left\|w\right\|,
        \end{equation*}
        where the first inequality comes from the fact that $\sigma_{\min}(M)=\min_{\|x\|=1}\|Mx\|$ for all square matrices $M$ (\cite[Theorem 3.1.2]{horn1991topics}) and the second inequality comes from Algorithm \ref{alg:mainrm} and the fact that $D^U_k$ must not be empty in this case (since $Q$ is specified).  Moreover, according to Lemma~\ref{lem:uniformlybdDU}, we have $\|g_Q\|\le\|D^U_k\|\|w\|\le M_{D^U}\|w\|$ and so $\|w\|\ge\|g_Q\|\slash M_{D^U}$.  Therefore,
        \begin{equation*}
            \left\|\left(D^U_k\right)^\top g_Q\right\| \ge \frac{\epsilon_{\mathrm{geo}}^2}{M_{D^U}}\left\|g_Q\right\|.
        \end{equation*}

        For the second term, we use the following inequality 
        \begin{equation*}
            \left\|\widetilde{R}^\top\right\| = \left\|\widetilde{R}\right\| = \left\|\widetilde{A}\right\| = \left\|(I_n-QQ^\top)A\right\| \le \left(\left\|I_n\right\|+\left\|QQ^\top\right\|\right)\left\|A\right\| \le 2M_{A}
        \end{equation*}
        to get
        \begin{equation*}
            \left\|\widetilde{R}^{-\top}A^\top g_{Q^\perp}\right\|=\frac{\left\|\widetilde{R}^\top\right\|\left\|\widetilde{R}^{-\top}A^\top g_{Q^\perp}\right\|}{\left\|\widetilde{R}^\top\right\|} \ge \frac{\left\|A^\top g_{Q^\perp}\right\|}{2M_{A}} \ge \frac{\alpha}{2M_{A}}\left\|g_{Q^\perp}\right\|.
        \end{equation*}

        Applying the above two lower bounds to \eqref{eq:DkTgexpression}, we have
        \begin{equation*}
            \left\|D_k^\top g\right\| \ge \sqrt{\frac{\epsilon_{\mathrm{geo}}^4}{\left(M_{D^U}\right)^2}\left\|g_Q\right\|^2 + \Delta_k^2\frac{\alpha^2}{4M_{A}^2}\left\|g_{Q^\perp}\right\|^2} \ge \alpha_D\sqrt{\left\|g_Q\right\|^2+\left\|g_{Q^\perp}\right\|^2} = \alpha_D\left\|g\right\|,
        \end{equation*}
        which completes the proof.

        $\hfill\qed$
    \end{proof}

\section{Convergence analysis}\label{sec:converg}
    In this section, we give the convergence analysis of the $\myalg$ algorithm.  Our proof follows a general framework in \cite{cartis2023scalable}.  However, since $\myalg$ uses different test conditions and a different class of models, our results are different from \cite{cartis2023scalable}. 

    We note that there exists $M_{D} > 0$ such that all $D_k$ satisfy $\|D_k\|\le M_{D}$.  To see this, suppose $D_k=[D_k^U~D_k^R]$ where $D_k^R$ is constructed by Algorithm \ref{alg:dirngen} with $\Delta=\Delta_k$.  Then we have $D_k^R=\Delta_k\widetilde{Q}$ where $\widetilde{Q}$ consists of orthonormal columns and so $\|D_k^R\|=\Delta_k$.  Hence,
    \begin{equation*}
        \left\|D_k\right\| \le \left\|D_k^U\right\| + \left\|D_k^R\right\| \le M_{D^U} + \Delta_k \le M_{D^U} + \Delta_{\max},
    \end{equation*}
    where $M_{D^U}$ is defined in Lemma \ref{lem:uniformlybdDU}.
    
    Suppose we have run $(K+1)$ iterations and the stopping condition is not triggered, i.e., $\Delta_k\ge\Delta_{\min}$ for all $k\in\mathcal{K}=\{0,...,K\}$.  We use the following subsets of $\mathcal{K}$:
    \begin{itemize}
        \item $\mathcal{A}:$ the subset of $\mathcal{K}$ where $D_k$ is $\alpha_D$-well aligned for $f$ at $x_k$.
        \item $\mathcal{A}^{\complement}$: the subset of $\mathcal{K}$ where $D_k$ is not $\alpha_D$-well aligned for $f$ at $x_k$.
        \item $\mathcal{C}:$ the subset of $\mathcal{K}$ where the criticality test condition is satisfied.
        \item $\mathcal{S}:$ the subset of $\mathcal{K}$ where the criticality test condition is not satisfied with $\rho_k\ge\eta_1$.
    \end{itemize}

    We begin our analysis by introducing a few lemmas following \cite{cartis2023scalable}.
    \begin{lemm}\label{lem:bdnormghat}
        Suppose $f\in\mathcal{C}^{2+}$ with Lipschitz constant $L_{\nabla^2 f}$.  Let $\epsilon>0$.  For all $k\in\mathcal{A}\backslash\mathcal{C}$ with $\|\nabla f(x_k)\|\ge\epsilon$, we have $$\left\|\nabla\widehat{m}_k(\mathbf{0})\right\|\ge\epsilon_g(\epsilon),$$
        where $$\epsilon_g(\epsilon)=\frac{\alpha_D\epsilon}{(\kappa_{eg}\mu+1)M_{D}}>0.$$
    \end{lemm}
    \begin{proof}
        The criticality test condition is not satisfied, so we have $\Delta_k\le\mu\|\nabla\widehat{m}_k(\mathbf{0})\|$.  Notice that since $\widehat{m}_k$ belongs to a class of $Q_k$-fully linear models of $f$ at $x_k$, we get
        \begin{align*}
            \left\|Q_k^\top\nabla f(x_k)\right\| &\le \left\|Q_k^\top\nabla f(x_k)-\nabla\widehat{m}_k(\mathbf{0})\right\| + \left\|\nabla\widehat{m}_k(\mathbf{0})\right\|\\
            &\le \kappa_{eg}\Delta_k + \left\|\nabla\widehat{m}_k(\mathbf{0})\right\|\\
            &\le \left(\kappa_{eg}\mu+1\right)\left\|\nabla\widehat{m}_k(\mathbf{0})\right\|.
        \end{align*}

        Since $D_k$ is $\alpha_D$-well aligned for $f$ at $x_k$, according to Lemma \ref{lem:wellalignedQR}, $Q_k$ is $(\alpha_D\slash\|D_k\|)$-well aligned for $f$ at $x_k$, so
        \begin{equation*}
            \left\|Q_k^\top\nabla f(x_k)\right\| \ge \frac{\alpha_D}{\left\|D_k\right\|}\left\|\nabla f(x_k)\right\| \ge \frac{\alpha_D}{M_{D}}\left\|\nabla f(x_k)\right\| \ge \frac{\alpha_D}{M_{D}}\epsilon.
        \end{equation*}
        The result follows from combining the above inequalities.

        $\hfill\qed$
    \end{proof}

    For the remainder of this section, we make use of the following assumptions.
    \begin{assump}\label{ass:smoothness}
        Suppose $f\in\mathcal{C}^{2+}$ with Lipschitz constant $L_{\nabla^2 f}$. 
    \end{assump}
    \begin{assump}\label{ass:bdbelow}
        Suppose $f$ is bounded below by $f_{low}$.
    \end{assump}
    \begin{assump}\label{ass:bdmodelHess}
        Suppose $\|\nabla^2\widehat{m}_k\|\le\kappa_H$ for all $k$, where $\kappa_H\ge 1$ is independent of $k$.
    \end{assump}
    \begin{assump}\label{ass:CauchyDec}
        Suppose all solutions $\widehat{s}_k$ of the trust-region subproblem satisfy $$\widehat{m}_k(\mathbf{0})-\widehat{m}_k(\widehat{s}_k)\ge c_1\left\|\nabla\widehat{m}_k(\mathbf{0})\right\|\min\left(\Delta_k, \frac{\|\nabla\widehat{m}_k(\mathbf{0})\|}{\max\left\{\left\|\nabla^2\widehat{m}_k\right\|, 1\right\}}\right),$$ for some $c_1\in (0, 1\slash 2]$ independent of $k$.
    \end{assump}
    \begin{assump}\label{ass:subspaceQual}
        Suppose each $D_k$ is $\alpha_D$-well aligned for $f$ at $x_k$ with probability at least $1-\delta_S$, where $\alpha_D$ is independent of $k$.
    \end{assump}

    We remark that under Assumption \ref{ass:smoothness}, as noted after Theorem \ref{thm:fullyquadconsts}, we have each $\widehat{m}_k$ belongs to a class of $Q_k$-fully linear models of $f$ at $x_k$ parameterized by $\Delta_k\in(0,\Delta_{\max}]$ with constants $\kappa_{ef}$ and $\kappa_{eg}$.  Assumption \ref{ass:bdbelow} is standard for optimization.  Using Assumption \ref{ass:smoothness}, Assumption \ref{ass:bdmodelHess} could alternatively be established by assuming $f$ has compact level sets.  Assumption \ref{ass:CauchyDec} is always achievable by finding a sufficiently high-quality solution to the trust-region subproblem \cite{audet2017derivative,conn2009introduction}.  Assumption \ref{ass:subspaceQual} can be satisfied by setting $p_{\mathrm{rand}}$, which is less than or equal to $q$, according to the lower bound on $q$ required by Theorem \ref{thm:alphaDwellalignedD}.

    \begin{lemm}\label{lem:countAcapS}
        Suppose Assumptions \ref{ass:smoothness} to \ref{ass:CauchyDec} hold.  Let $\epsilon>0$.  If $\|\nabla f(x_k)\|\ge\epsilon$ for all $k\in\mathcal{K}$, then $$\left|\mathcal{A}\cap\mathcal{S}\right|\le\phi(\epsilon),$$ where $$\phi(\epsilon)=\frac{f(x^0)-f_{low}}{\eta_1c_1\epsilon_g(\epsilon)\min(\epsilon_g(\epsilon)\slash\kappa_H, \Delta_{\min})}$$ and $\epsilon_g(\epsilon)$ is defined in Lemma \ref{lem:bdnormghat}.
    \end{lemm}
    \begin{proof}
        See \cite[Lemma 4]{cartis2023scalable} with $\Delta=\Delta_{\min}$.

        $\hfill\qed$
    \end{proof}

    The following lemma bounds the remaining part of $\mathcal{A}$.  We note that the proof is inspired by \cite[Lemma 6]{cartis2023scalable}.
    \begin{lemm}\label{lem:countAslashS}
        The numbers of iterations in $\mathcal{A}\backslash\mathcal{S}$ and $\mathcal{S}$ have the relation $$\left|\mathcal{A}\backslash\mathcal{S}\right|\le C_1\left|\mathcal{S}\right|+C_2,$$ where $C_1=\ln(1\slash\gamma_{\mathrm{dec}})^{-1}\ln(\gamma_{\mathrm{inc}})$ and $C_2=\ln(1\slash\gamma_{\mathrm{dec}})^{-1}\ln(\Delta_0\slash\Delta_{\min}).$
    \end{lemm}
    \begin{proof}
        For all $k\in\mathcal{A}\backslash\mathcal{S}$, we have either the criticality test condition is satisfied, or the criticality test condition is not satisfied with $\rho_k<\eta_1$. 
        In both cases, we have $\Delta_{k+1}=\gamma_{\mathrm{dec}}\Delta_k$.  For all $k\in\mathcal{S}$, we have $\Delta_{k+1}\le\gamma_{\mathrm{inc}}\Delta_k$.

        Notice that $\Delta_k$ can only increase when $k\in\mathcal{S}$.  Hence, the maximum $\Delta_k$ for all $k\in\mathcal{K}$ is bounded by $\Delta_0\gamma_{\mathrm{inc}}^{\left|\mathcal{S}\right|}$.  Since the stopping condition is not triggered for all $k\in\mathcal{K}$, we must have %$$\Delta_0\prod_{k\in\mathcal{S}}\frac{\Delta_{k+1}}{\Delta_k}\prod_{k\in\mathcal{A}\backslash\mathcal{S}}\frac{\Delta_{k+1}}{\Delta_k}\ge\Delta_{\min}$$
        %and obtain 
        $$\Delta_0\gamma_{\mathrm{inc}}^{\left|\mathcal{S}\right|}\gamma_{\mathrm{dec}}^{\left|\mathcal{A}\backslash\mathcal{S}\right|}\ge\Delta_{\min}.$$
        Taking the natural logarithm of both sides, we get $$\ln(\Delta_0)+\ln(\gamma_{\mathrm{inc}})\left|\mathcal{S}\right|+\ln(\gamma_{\mathrm{dec}})\left|\mathcal{A}\backslash\mathcal{S}\right|\ge\ln(\Delta_{\min}),$$ i.e., 
        \begin{align*}
            \left|\mathcal{A}\backslash\mathcal{S}\right|\le\ln(\frac{1}{\gamma_{\mathrm{dec}}})^{-1}\ln(\gamma_{\mathrm{inc}})\left|\mathcal{S}\right|+\ln(\frac{1}{\gamma_{\mathrm{dec}}})^{-1}\ln(\frac{\Delta_0}{\Delta_{\min}}).
        \end{align*}

        $\hfill\qed$
    \end{proof}

    Combining Lemma \ref{lem:countAcapS} and \ref{lem:countAslashS}, we get an upper bound for the cardinality of $\mathcal{A}$.
    \begin{lemm}\label{lem:countA}
        Suppose Assumptions \ref{ass:smoothness} to \ref{ass:CauchyDec} hold.  Let $\epsilon>0$.  If $\|\nabla f(x_k)\|\ge\epsilon$ for all $k\in\mathcal{K}$, then we have $$\left|\mathcal{A}\right|\le\psi(\epsilon)+\frac{C_1(K+1)}{C_1+1},$$ where $\psi(\epsilon)=\phi(\epsilon)+C_2\slash (C_1+1)$, $\phi(\epsilon)$ is defined in Lemma \ref{lem:countAcapS}, and $C_1$ and $C_2$ are defined in Lemma~\ref{lem:countAslashS}.
    \end{lemm}
    \begin{proof}
        Using the Lemma \ref{lem:countAcapS} and \ref{lem:countAslashS}, we obtain
        \begin{align*}
            \left|\mathcal{A}\right|&=\left|\mathcal{A}\cap\mathcal{S}\right|+\left|\mathcal{A}\backslash\mathcal{S}\right|\\
            &\le \phi(\epsilon)+C_1\left|\mathcal{S}\right|+C_2\\
            &= \phi(\epsilon)+C_1\left(\left|\mathcal{A}\cap\mathcal{S}\right|+\left|\mathcal{A}^{\complement}\cap\mathcal{S}\right|\right)+C_2\\
            &= \phi(\epsilon)+C_1\left|\mathcal{A}\cap\mathcal{S}\right|+C_1\left|\mathcal{A}^{\complement}\cap\mathcal{S}\right|+C_2\\
            &\le (C_1+1)\phi(\epsilon)+C_1\left|\mathcal{A}^{\complement}\right|+C_2\\
            &= (C_1+1)\phi(\epsilon)+C_1\left((K+1)-\left|\mathcal{A}\right|\right)+C_2\\
            &= (C_1+1)\phi(\epsilon)+C_1(K+1)-C_1\left|\mathcal{A}\right|+C_2,
        \end{align*}
        which gives
        \begin{align*}
            \left|\mathcal{A}\right|\le\phi(\epsilon)+\frac{C_1(K+1)+C_2}{C_1+1}=\psi(\epsilon)+\frac{C_1(K+1)}{C_1+1}.
        \end{align*}

        $\hfill\qed$
    \end{proof}

    The next Lemma is based on a general framework introduced in \cite{gratton2015direct}. 
    \begin{lemm}\label{lem:probcountA}
        Suppose Assumption \ref{ass:subspaceQual} holds.  Then for all $\delta\in (0,1)$, we have $$\mathbb{P}\left[\left|\mathcal{A}\right|\le(1-\delta)(1-\delta_S)(K+1)\right]\le e^{-\delta^2(1-\delta_S)(K+1)\slash 2}.$$
    \end{lemm}
    \begin{proof}
        See the proof of \cite[Lemma 10 until inequality (67)]{cartis2023scalable}.  Alternatively, \cite[Lemma 4.5]{gratton2015direct} with $p=1-\delta_S,\lambda=(1-\delta)(1-\delta_S)$, and $k=K+1$.

        $\hfill\qed$
    \end{proof}

    Now we present our first convergence result, which gives a lower bound on the probability of finding a sufficiently small gradient in the first $(K+1)$ iterations.
    \begin{theo}\label{thm:converg_finite}
        Suppose Assumptions \ref{ass:smoothness} to \ref{ass:subspaceQual} hold.  Suppose $0<\delta_S<1\slash(C_1+1)$ with $C_1$ defined in Lemma \ref{lem:countAslashS}.  Let $\epsilon>0$ and $$K\ge\frac{2\psi(\epsilon)}{1-\delta_S-C_1\slash\left(C_1+1\right)}-1,$$ where $\psi(\epsilon)$ is defined in Lemma \ref{lem:countA}.  If the stopping condition is not triggered for all $k\in\{0,...,K\}$, then we have $$\mathbb{P}\left[\min\limits_{k\le K}\left\|\nabla f(x_k)\right\|\le\epsilon\right]\ge 1-e^{-C_3(K+1)},$$
        where $$C_3=\frac{\left(1-\delta_S-C_1\slash\left(C_1+1\right)\right)^2}{8(1-\delta_S)}.$$
    \end{theo}
    \begin{proof}
        Let $\epsilon_K=\min_{k\le K}\|\nabla f(x_k)\|$ and $A_K=\left|\mathcal{A}\right|$.  If $\epsilon_K=0$, then $\min_{k\le K}\|\nabla f(x_k)\|\le\epsilon$ and we are done.  If $\epsilon_K>0$, then we have $A_K\le\psi(\epsilon_K)+C_1(K+1)\slash(C_1+1)$ by Lemma \ref{lem:countA}.  Notice that $\psi(\cdot)$ is a non-increasing function and $$\psi(\epsilon)\le\frac{1}{2}\left(1-\delta_S-\frac{C_1}{C_1+1}\right)\left(K+1\right).$$ Therefore, we have
        \begin{align*}
            \mathbb{P}\left[\epsilon_K\ge\epsilon\right]&\le \mathbb{P}\left[\psi(\epsilon_K)\le\psi(\epsilon)\right]\\
            &\le \mathbb{P}\left[\psi(\epsilon_K)\le\frac{1}{2}\left(1-\delta_S-\frac{C_1}{C_1+1}\right)\left(K+1\right)\right]\\
            &\le \mathbb{P}\left[A_K-\frac{C_1(K+1)}{C_1+1}\le\frac{1}{2}\left(1-\delta_S-\frac{C_1}{C_1+1}\right)\left(K+1\right)\right]\\
            &= \mathbb{P}\left[A_K\le\frac{1}{2}\left(1-\delta_S+\frac{C_1}{C_1+1}\right)\left(K+1\right)\right].
        \end{align*}

        Since $\delta_S<1\slash(C_1+1)$ and $C_1>0$, we have $1-\delta_S>C_1\slash(C_1+1)>0$, which gives $$\frac{C_1}{(C_1+1)(1-\delta_S)}\in (0,1).$$ Hence, we can take $$\delta=\frac{1}{2}\left[1-\frac{C_1}{(C_1+1)(1-\delta_S)}\right]\in (0,1)$$ in Lemma \ref{lem:probcountA} and obtain
        \begin{align*}
            \mathbb{P}\left[A_K\le\frac{1}{2}\left(1-\delta_S+\frac{C_1}{C_1+1}\right)\left(K+1\right)\right]\le e^{-C_3(K+1)}.
        \end{align*}
        Therefore, we have 
        \begin{align*}
            \mathbb{P}\left[\min\limits_{k\le K}\left\|\nabla f(x_k)\right\|\le\epsilon\right]=1-\mathbb{P}\left[\epsilon_K\ge\epsilon\right]\ge 1-e^{-C_3(K+1)}.
        \end{align*}

        $\hfill\qed$
    \end{proof}

    Based on Theorem \ref{thm:converg_finite}, we get almost-sure liminf-type of convergence.
    \begin{theo}\label{thm:converg_liminf}
        Suppose Assumptions \ref{ass:smoothness} to \ref{ass:subspaceQual} hold.  Suppose $0<\delta_S<1\slash(C_1+1)$ with $C_1$ defined in Lemma \ref{lem:countAslashS}.  If $\myalg$ is run with $\Delta_{\min}=0$, then $$\mathbb{P}\left[\inf\limits_{k\ge 0}\left\|\nabla f(x_k)\right\|=0\right]=1.$$ %In particular, if $\|\nabla f(x_k)\|>0$ for all $k$, then $$\mathbb{P}\left[\liminf\limits_{k\ge 0}\left\|\nabla f(x_k)\right\|=0\right]=1.$$
    \end{theo}
    \begin{proof}
        Let $\epsilon>0$ and
        \begin{equation}\label{ineq:lowerbdK}
            K\ge\frac{2\psi(\epsilon)}{1-\delta_S-C_1\slash\left(C_1+1\right)}-1,
        \end{equation}
        where $\psi(\epsilon)$ is defined in Lemma \ref{lem:countA}.  Notice that for all $\epsilon>0$ and $K$ satisfying \eqref{ineq:lowerbdK} we have
        \begin{equation*}
            \mathbb{P}\left[\inf\limits_{k\ge 0}\left\|\nabla f(x_k)\right\|\le\epsilon\right]\ge\mathbb{P}\left[\min\limits_{k\le K}\left\|\nabla f(x_k)\right\|\le\epsilon\right]\ge 1-e^{-C_3(K+1)},
        \end{equation*}
        where the second inequality comes from Theorem \ref{thm:converg_finite}.  Taking $K\to\infty$, we have for all $\epsilon>0$, 
        \begin{equation*}
            \mathbb{P}\left[\inf\limits_{k\ge 0}\left\|\nabla f(x_k)\right\|\le\epsilon\right] \ge 1-\lim\limits_{K\to\infty}e^{-C_3(K+1)} = 1.
        \end{equation*}
        Finally, we take $\epsilon\to 0$ and use the continuity of probability to get 
        \begin{align*}
            \mathbb{P}\left[\inf\limits_{k\ge 0}\left\|\nabla f(x_k)\right\|=0\right] = 1.
        \end{align*}

        $\hfill\qed$
    \end{proof}
    
    We also have the expectation on the number of iterations used to find a sufficiently small gradient.
    \begin{theo}\label{thm:converg_expectation}
        Suppose Assumptions \ref{ass:smoothness} to \ref{ass:subspaceQual} hold.  Suppose $0<\delta_S<1\slash(C_1+1)$ with $C_1$ defined in Lemma \ref{lem:countAslashS}.  Let $\epsilon>0$ and $K_\epsilon=\min\left\{k:\|\nabla f(x_k)\|\le\epsilon\right\}$.  Denote $$K_{\min}(\epsilon)=\frac{2\psi(\epsilon)}{1-\delta_S-C_1\slash\left(C_1+1\right)}-1,$$ where $\psi(\epsilon)$ is defined in Lemma \ref{lem:countA}.  Then we have $$\mathbb{E}\left[K_\epsilon\right]\le K_{\min}(\epsilon) + \frac{e^{-C_3(K_{\min}(\epsilon)+1)}}{1-e^{-C_3}} = \mathcal{O}(\epsilon^{-2}),$$ where $C_3$ is defined in Theorem \ref{thm:converg_finite}.
    \end{theo}
    \begin{proof}
        From Theorem \ref{thm:converg_finite}, we have that for all $K\ge K_{\min}(\epsilon)$,
        \begin{equation*}
            \mathbb{P}\left[K_\epsilon\le K\right]=\mathbb{P}\left[\min_{k\le K}\|f(x_k)\|\le\epsilon\right]\ge 1-e^{-C_3(K+1)}.
        \end{equation*}
        Notice that $\mathbb{E}[X]=\sum_{k=0}^\infty\mathbb{P}\left[X > k\right]$ for all non-negative integer-valued random variables $X$.  Thus we obtain
        \begin{align*}
            \mathbb{E}\left[K_\epsilon\right]&=\sum\limits_{K=0}^{K_{\min}(\epsilon)-1}\mathbb{P}\left[K_\epsilon > K\right] + \sum\limits_{K=K_{\min}(\epsilon)}^\infty\mathbb{P}\left[K_\epsilon > K\right]\\
            &\le K_{\min}(\epsilon) + \sum\limits_{K=K_{\min}(\epsilon)}^\infty e^{-C_3(K+1)}\\
            &= K_{\min}(\epsilon) + \frac{e^{-C_3(K_{\min}(\epsilon)+1)}}{1-e^{-C_3}}.
        \end{align*}
        The proof is complete by noticing that $K_{\min}(\epsilon)=\mathcal{O}(\epsilon^{-2})$.

        $\hfill\qed$
    \end{proof}

\section{Numerical experiments}\label{sec:numexp}
    In this section, we develop numerical experiments to explore the performance of $\myalg$.  We select two sets of test problems from the CUTEst collection~\cite{gould2015cutest}.  The dimension of each problem is between $n=1000$ and $n=1082$.  In the first set, we select 73 unconstrained optimization problems with various forms of objective functions.  This test set is chosen in order to compare the performance of $\myalg$ based on linear and quadratic approximation models.  In the second set, we pick 32 unconstrained nonlinear least-squares problems, i.e., 
    \begin{equation*}
        \min\limits_{x\in\mathbb{R}^n}f(x)=\frac{1}{2}\left\|\bm{g}(x)\right\|^2=\frac{1}{2}\sum\limits_{i=1}^m g_i(x)^2,
    \end{equation*}
    where $\bm{g}(x)=[g_1(x),...,g_m(x)]^\top$ and $g_i(x):\mathbb{R}^n\to\mathbb{R}$ for all $i=1,...,m$.  This set is chosen in order to demonstrate the advantages of exploring the structure of objective functions, which is a technique used in \cite{cartis2023scalable}.

    Our implementation of $\myalg$ is in Python 3, with random library {\tt numpy.random} and trust-region subproblem solver {\tt trsbox} from the DFBGN algorithm implemented in \cite{cartis2023scalable}.  The code is available upon request.
    % The code is available on Github\footnote{\url{https://github.com/yiwchen233/QARSTA}}.

    \subsection{Results on general unconstrained problems}\label{subsec:numexpgeneral}
    In order to demonstrate the advantages and disadvantages of using quadratic models, we compare the performance of $\myalg$ based on three sorts of models:
    \begin{enumerate}
        \item determined quadratic models $\widehat{m}(\widehat{s})$;
        \item underdetermined quadratic models, constructed by a formula similar to $\widehat{m}(\widehat{s})$, with the model Hessian $\nabla_S^2\widehat{f}(\mathbf{0};R)$ replaced by
            \begin{equation*}
                R^{-\top}\widetilde{\delta}_{\delta_f}(\mathbf{0};R)R^{-1},
            \end{equation*}
            where the matrix $\widetilde{\delta}_{\delta_f}(\mathbf{0};R)$ is given by setting all off-diagonal elements of $\delta_{\delta_f}(\mathbf{0};R)$ to zero (see Equation \eqref{eq:deltaofdeltas} for definition of the $\delta_{\delta_f}$ operator);
        \item linear models, constructed by the formula $$\widetilde{m}(\widehat{s})=\widehat{f}(\mathbf{0})+\nabla_S \widehat{f}(\mathbf{0};R)^\top\widehat{s},$$ which is the linear interpolation model of $\widehat{f}$ on the $(p+1)$ points $\{\mathbf{0}\}\cup\{r_i:i=1,...,p\}$.
    \end{enumerate}
    
    The underdetermined quadratic model only uses function values in the form of $\widehat{f}(\mathbf{0}),\widehat{f}(r_i)$, and $\widehat{f}(2r_i)$ where $i=1,...,p$.  Using a similar analysis technique as Theorems \ref{thm:fullspaceunderd} and~\ref{thm:subspaced}, it can be shown that the model constructed this way interpolates $\widehat{f}$ on all $(2p+1)$ points $\{\mathbf{0}\}\cup\{r_i:i=1,...,p\}\cup\{2r_i:i=1,...,p\}$.  
    
    We note that all three models satisfy the convergence analysis shown in Section~\ref{sec:converg}.  Indeed, for determined quadratic models, we have proved in Subsection \ref{subsec:modelAcc} that they are $Q_k$-fully quadratic and therefore $Q_k$-fully linear.  For the other two models, notice that a full space linear interpolation model belongs to {\it a class of fully linear models} \footnote{In $\mathbb{R}^n$, this can be viewed as a class of $I_n$-fully linear models.} \cite{audet2017derivative,conn2009introduction}.  Since the underdetermined quadratic and linear models are at least as accurate as linear interpolation models, following a similar proof as in \cite{audet2017derivative,conn2009introduction} and Theorem \ref{thm:fullyquadconsts} in this paper, it can be shown that they are $Q_k$-fully linear, which is sufficient for the results in Section \ref{sec:converg}.
    
    For each of the three models, we consider $(p,p_{\mathrm{rand}})\in\{(1,1), (10,1), (10,3), (10,10)\}$.  Hence, there are in total 12 variants of $\myalg$ to be compared.  We note that among all 12 solvers, only the 6 solvers that have $p=10$ and $p_{\mathrm{rand}}\in\{3, 10\}$ are guaranteed to converge by our analysis in Section \ref{sec:converg}.  Indeed, according to Theorems \ref{thm:alphaDwellalignedD} to \ref{thm:converg_expectation}, we need $q\ge 4(1-\alpha)^{-2}\ln(1\slash\delta_S)$ and $0<\delta_S<1\slash(C_1+1)$, where $\alpha\in(0,1)$, $q\ge p_{\mathrm{rand}}$ is the number of columns in $D_k^R$, and $C_1$ is defined in Lemma \ref{lem:countAslashS}.  In our implementation, we always set $\gamma_{\mathrm{dec}}=1\slash 2$ and $\gamma_{\mathrm{inc}}=2$, which imply $C_1=1$.  Letting $p_{\mathrm{rand}}\ge 4(1-\alpha)^{-2}\ln(1\slash\delta_S)$ with $0<\alpha<1$ and $0<\delta_S<1\slash 2$, we get a lower bound $p_{\mathrm{rand}}>4\ln 2\approx 2.77$.  Hence, the solvers with $p_{\mathrm{rand}}\ge 3$ are guaranteed to converge by our convergence analysis.  For solvers with $p_{\mathrm{rand}}< 3$, our numerical results show that they still have a high probability of converging. 

    We use performance profiles \cite{more2009benchmarking} and data profiles \cite{more2009benchmarking} to present the results.  The performance measure is either the number of function evaluations or runtime.  For each test problem, we let $x^*$ be the point that has the minimum objective function value obtained by all solvers within $100(n+1)$ function evaluations.  We regard a problem as successfully solved with accuracy level $\tau\in (0,1)$ by a solver if, within $100(n+1)$ function evaluations, there exists an iterate $x_k$ such that $$f(x_k)\le f(x^*)+\tau\left(f(x_0)-f(x^*)\right),$$ where $x_0$ is the starting point.  In the performance profile, we plot the proportion of test problems successfully solved by a solver against the ratio between the performance measure of that solver and the best performance measure of all solvers.  In the data profiles, we plot the proportion of test problems successfully solved by a solver against the performance measure of that solver.  For detailed explanations, see \cite{more2009benchmarking}.
    
    For clarity, for all profiles, we use the legend shown in Figure \ref{fig:legend}.  We note that the data profiles based on runtime are truncated at 1 hour.
    \begin{figure}[!htb]
         \centering
         \includegraphics[width=0.4\textwidth]{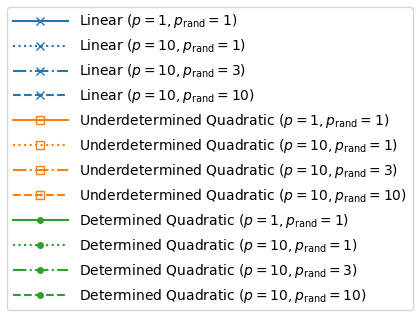}
         \caption{Legend of figures.}
         \label{fig:legend}
    \end{figure}\FloatBarrier

    \subsubsection{Results based on runtime}
    Figures \ref{fig:combinedPerf05&1e-3rt} and \ref{fig:combinedData05&1e-3rt} present the performance and data profiles based on runtime with accuracy levels $\tau=0.5$ and $\tau=10^{-3}$.  Figures \ref{fig:combinedPerf05&1e-3rt} and \ref{fig:combinedData05&1e-3rt} show that, for the low accuracy level $\tau=0.5$, when measured by the runtime, solvers with $p=p_{\mathrm{rand}}=1$ perform better than other choices of $p$ and $p_{\mathrm{rand}}$.  The solver based on determined quadratic models with $p=p_{\mathrm{rand}}=1$ is the best among all solvers.  Moreover, solvers based on both sorts of quadratic models perform better than solvers based on linear models in general.  For the high accuracy level $\tau=10^{-3}$, we see a similar pattern when the runtime is small, but the solver based on underdetermined quadratic models with $p=10$ and $p_{\mathrm{rand}}=10$ and linear models with $p=10$ and $p_{\mathrm{rand}}=3$ solve more problem as runtime increases.  As shown in Figure~\ref{fig:combinedData05&1e-3rt}, within the 1-hour runtime budget, the solver based on underdetermined quadratic models with $p=10$ and $p_{\mathrm{rand}}=10$ finally solves the most percentage of problems.  
    
    This implies that solvers based on quadratic models are generally more efficient than solvers based on linear models, and solvers with very small $p$ and $p_{\mathrm{rand}}$ have the best efficiency.  However, if the runtime is long enough, then solvers satisfying the convergence analysis have the potential to solve more problems.  We should also note that our objective functions are very inexpensive to evaluate.  Therefore, solvers based on quadratic models tend to have better performance on optimization problems where the objective function evaluations are inexpensive and quick solutions are preferred.

    For completeness, the results with $\tau=10^{-1}$ and $\tau=10^{-2}$ can be found in Appendix \ref{app:morenumexp}.
    \begin{figure}[!htb]
         \begin{subfigure}{0.49\textwidth}
             \centering
             \includegraphics[width=\textwidth]{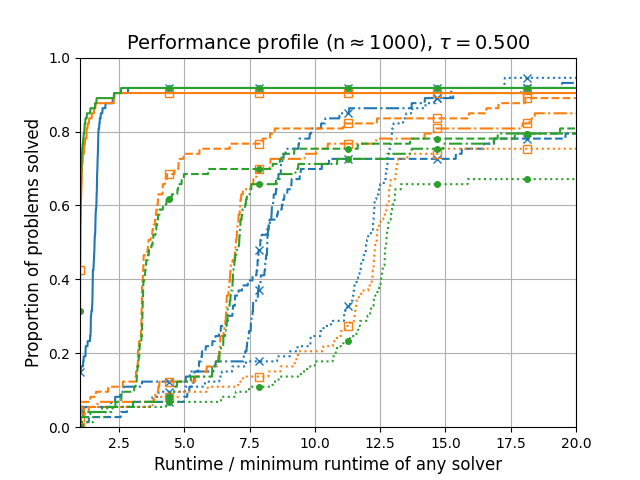}
         \end{subfigure}
         \hfill
         \begin{subfigure}{0.49\textwidth}
             \centering
             \includegraphics[width=\textwidth]{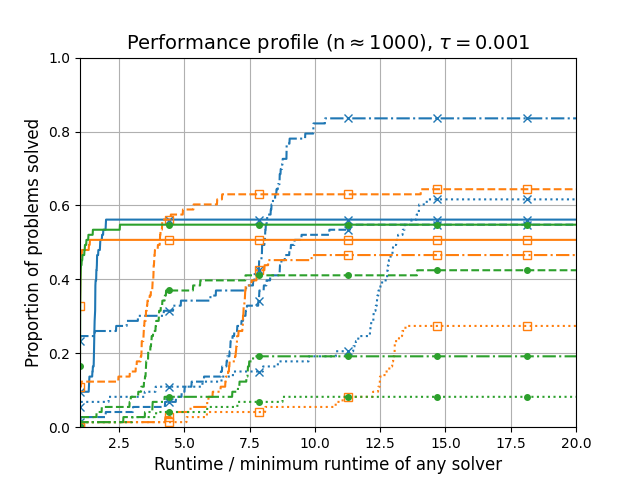}
         \end{subfigure}
         \caption{Performance profiles with $\tau=0.5$ and $\tau=10^{-3}$ comparing the runtime.}
         \label{fig:combinedPerf05&1e-3rt}
    \end{figure}\FloatBarrier
    
    \begin{figure}[!htb]
         \begin{subfigure}{0.49\textwidth}
             \centering
             \includegraphics[width=\textwidth]{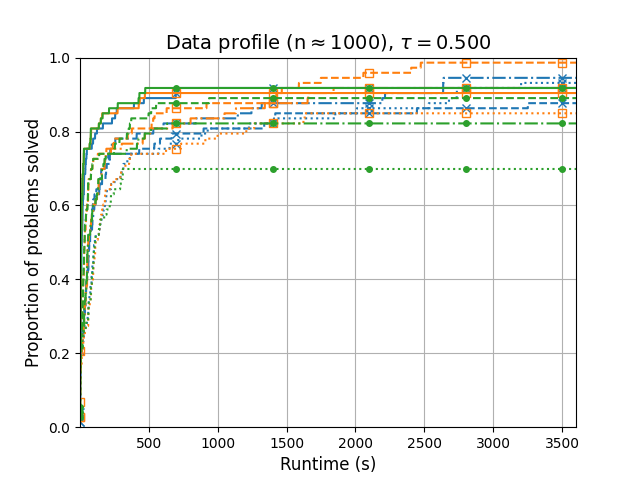}
         \end{subfigure}
         \hfill
         \begin{subfigure}{0.49\textwidth}
             \centering
             \includegraphics[width=\textwidth]{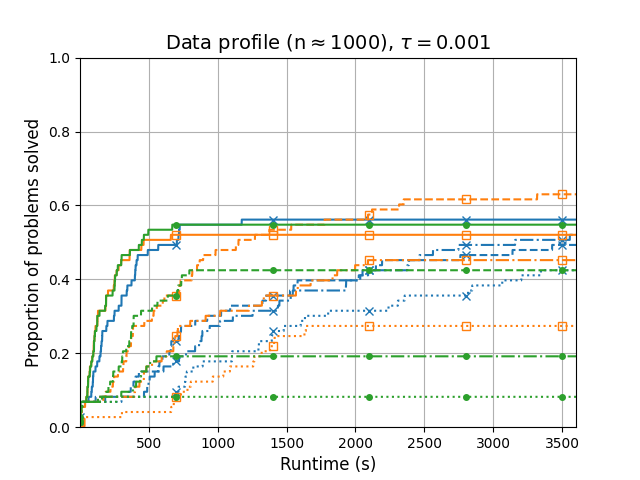}
         \end{subfigure}
         \caption{Data profiles with $\tau=0.5$ and $\tau=10^{-3}$ comparing the runtime.}
         \label{fig:combinedData05&1e-3rt}
    \end{figure}\FloatBarrier

    \subsubsection{Results based on function evaluations}
    Figures \ref{fig:combinedPerf05&1e-3evals} and \ref{fig:combinedData05&1e-3evals} present the performance and data profiles based on function evaluations with accuracy levels $\tau=0.5$ and $\tau=10^{-3}$.  We can see that when measured by function evaluations, for both low accuracy level $\tau=0.5$ and high accuracy level $\tau=10^{-3}$, solvers based on linear models perform better than other solvers in general, and the performance of solvers based on underdetermined models is better than solvers based on determined quadratic models in general.  
    
    This implies that although quadratic models are more accurate than linear models and thus may reduce the number of iterations required to reach a certain accuracy level, the increased number of function evaluations they need still makes the solvers take more overall function evaluations.  
    
    For completeness, the results with $\tau=10^{-1}$ and $\tau=10^{-2}$ can be found in Appendix \ref{app:morenumexp}.
    \begin{figure}[!htb]
         \begin{subfigure}{0.49\textwidth}
             \centering
             \includegraphics[width=\textwidth]{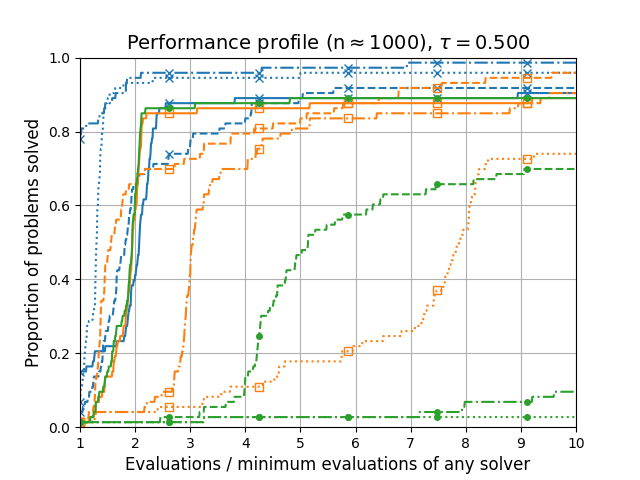}
         \end{subfigure}
         \hfill
         \begin{subfigure}{0.49\textwidth}
             \centering
             \includegraphics[width=\textwidth]{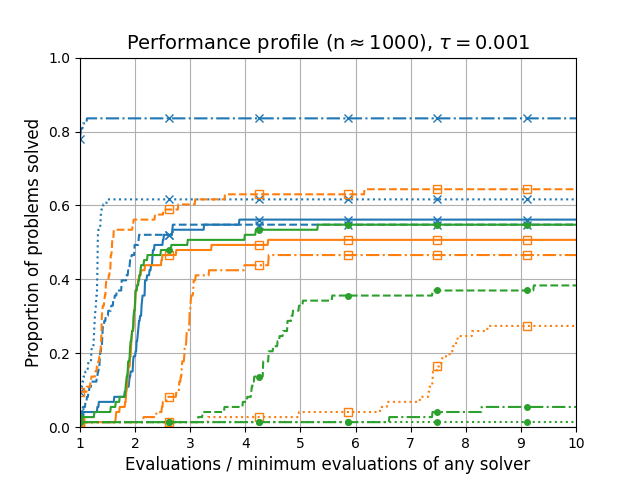}
         \end{subfigure}
         \caption{Performance profiles with $\tau=0.5$ and $\tau=10^{-3}$ comparing the function evaluations.}
         \label{fig:combinedPerf05&1e-3evals}
    \end{figure}\FloatBarrier
    
    \begin{figure}[!htb]
         \begin{subfigure}{0.49\textwidth}
             \centering
             \includegraphics[width=\textwidth]{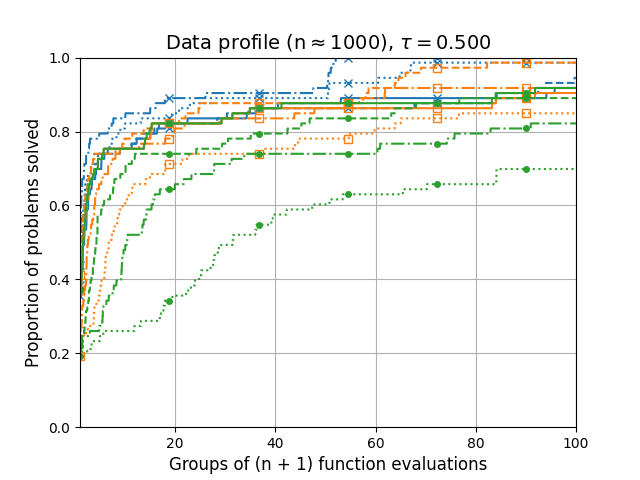}
         \end{subfigure}
         \hfill
         \begin{subfigure}{0.49\textwidth}
             \centering
             \includegraphics[width=\textwidth]{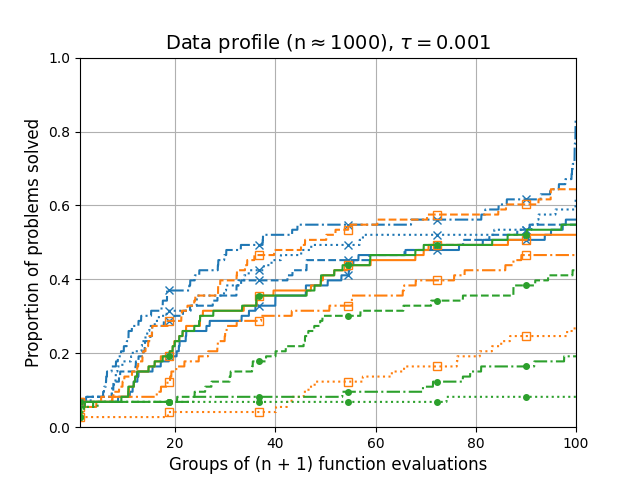}
         \end{subfigure}
         \caption{Data profiles with $\tau=0.5$ and $\tau=10^{-3}$ comparing the function evaluations.}
         \label{fig:combinedData05&1e-3evals}
    \end{figure}\FloatBarrier

    \subsection{Results on nonlinear least-squares problems}
    For nonlinear least-squares problems, there exists another method to construct quadratic models.  As shown in \cite{cartis2023scalable}, we can first construct a subspace linear interpolation model for $\bm{g}$, denoted by $\widetilde{m}_{\bm{g}}(\widehat{s})$, which interpolates $\widehat{\bm{g}}(\widehat{s})=\bm{g}(x^0+Q\widehat{s})$ on the $(p+1)$ points $\{\mathbf{0}\}\cup\{r_i:i=1,...,p\}$.  Then we compute the square of the norm of $\widetilde{m}_{\bm{g}}(\widehat{s})$ to get a quadratic model for $\widehat{f}$.  That is, $$\widehat{f}(\widehat{s})\approx\frac{1}{2}\left\|\widetilde{m}_{\bm{g}}(\widehat{s})\right\|^2.$$

    We note that the quadratic models constructed by this method may not interpolate $\widehat{f}$ on more than $(p+1)$ points.  However, it can be shown that they are $Q_k$-fully linear and therefore satisfy the convergence analysis in Section \ref{sec:converg}.  The proof can be found in \cite{cartis2023scalable} or by applying the theory of model functions operations developed in \cite{chen2022error}.

    In this set of experiments, we compare the performance of $\myalg$ based on the model mentioned above and the three models described in Subsection \ref{subsec:numexpgeneral}.  The results are presented by performance profiles \cite{more2009benchmarking} and data profiles \cite{more2009benchmarking} with the same parameters as in Subsection \ref{subsec:numexpgeneral}.  For all the profiles, we use the legend shown in Figure \ref{fig:legend} plus the following new labels.
    \begin{figure}[!htb]
         \centering
         \includegraphics[width=0.36\textwidth]{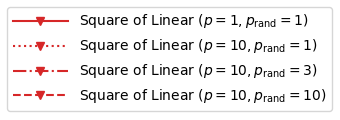}
         \caption{Legend of figures (continued).}
         \label{fig:legendcont}
    \end{figure}\FloatBarrier

    \subsubsection{Results based on runtime}
    Figures \ref{fig:LSPerfwSL05&1e-3rt} and \ref{fig:LSDatawSL05&1e-3rt} present the performance and data profiles based on runtime with accuracy levels $\tau=0.5$ and $\tau=10^{-3}$.  Figures \ref{fig:LSPerfwSL05&1e-3rt} and \ref{fig:LSDatawSL05&1e-3rt} show that, for the low accuracy level $\tau=0.5$, when measured by the runtime, solvers based on square-of-linear models perform better than linear models, but worse than other quadratic models in general.  However, for the high accuracy level $\tau=10^{-3}$, the solver based on square-of-linear models with $p=p_{\mathrm{rand}}=10$ is significantly better than other solvers, which agrees with past results found in \cite{cartis2023scalable}.  This implies that exploring the structure of objective functions increases the efficiency of solvers, especially for high accuracy levels.

    For completeness, the results with $\tau=10^{-1}$ and $\tau=10^{-2}$ can be found in Appendix \ref{app:morenumexp}.
    \begin{figure}[!htb]
         \begin{subfigure}{0.49\textwidth}
             \centering
             \includegraphics[width=\textwidth]{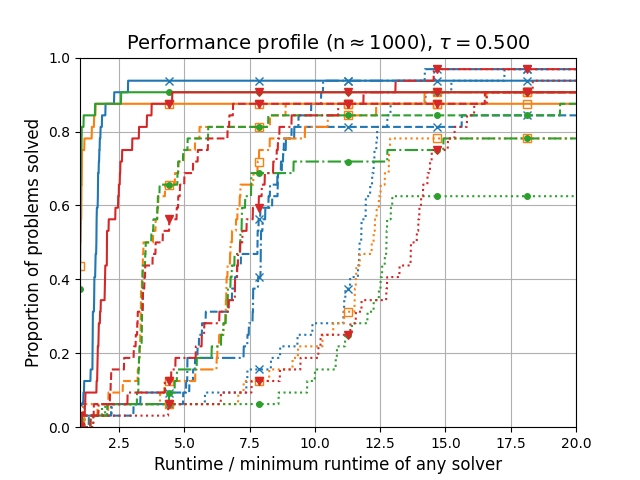}
         \end{subfigure}
         \hfill
         \begin{subfigure}{0.49\textwidth}
             \centering
             \includegraphics[width=\textwidth]{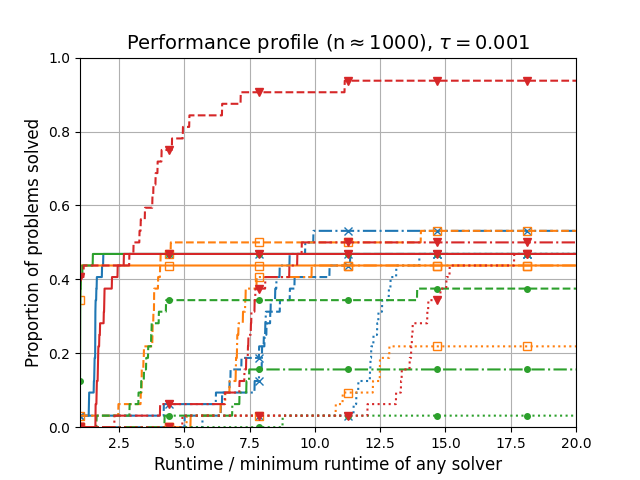}
         \end{subfigure}
         \caption{Performance profiles with $\tau=0.5$ and $\tau=10^{-3}$ comparing the runtime.}
         \label{fig:LSPerfwSL05&1e-3rt}
    \end{figure}\FloatBarrier
    
    \begin{figure}[!htb]
         \begin{subfigure}{0.49\textwidth}
             \centering
             \includegraphics[width=\textwidth]{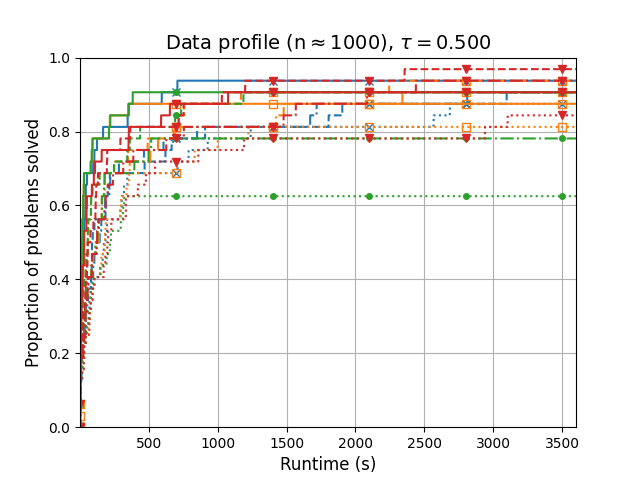}
         \end{subfigure}
         \hfill
         \begin{subfigure}{0.49\textwidth}
             \centering
             \includegraphics[width=\textwidth]{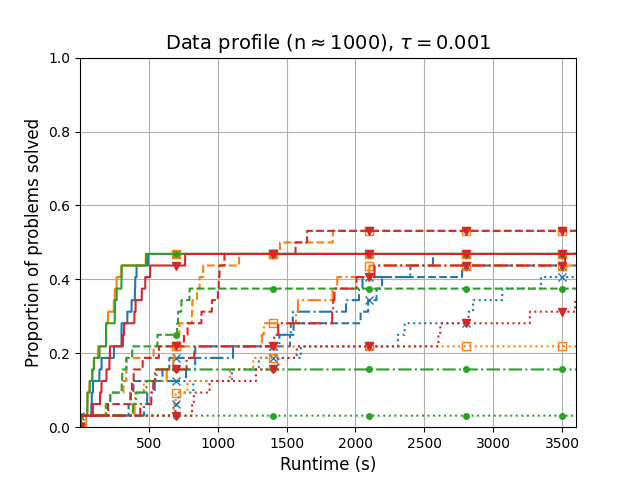}
         \end{subfigure}
         \caption{Data profiles with $\tau=0.5$ and $\tau=10^{-3}$ comparing the runtime.}
         \label{fig:LSDatawSL05&1e-3rt}
    \end{figure}\FloatBarrier

    \subsubsection{Results based on function evaluations}
    Figures \ref{fig:LSPerfwSL05&1e-3evals} and \ref{fig:LSDatawSL05&1e-3evals} present the performance and data profiles based on function evaluations with accuracy levels $\tau=0.5$ and $\tau=10^{-3}$.  When measured by function evaluations, as shown in Figures \ref{fig:LSPerfwSL05&1e-3evals} and \ref{fig:LSDatawSL05&1e-3evals}, for both accuracy levels $\tau=0.5$ and $\tau=10^{-3}$, the solvers based on square-of-linear models perform the best in general.  Once again, the solver based on square-of-linear models with $p=p_{\mathrm{rand}}=10$ is significantly better than other solvers and agrees with past results found in \cite{cartis2023scalable}.  This again demonstrates the advantage of exploring the structure of objective functions.
    
    For completeness, the results with $\tau=10^{-1}$ and $\tau=10^{-2}$ can be found in Appendix \ref{app:morenumexp}.
    \begin{figure}[!htb]
         \begin{subfigure}{0.49\textwidth}
             \centering
             \includegraphics[width=\textwidth]{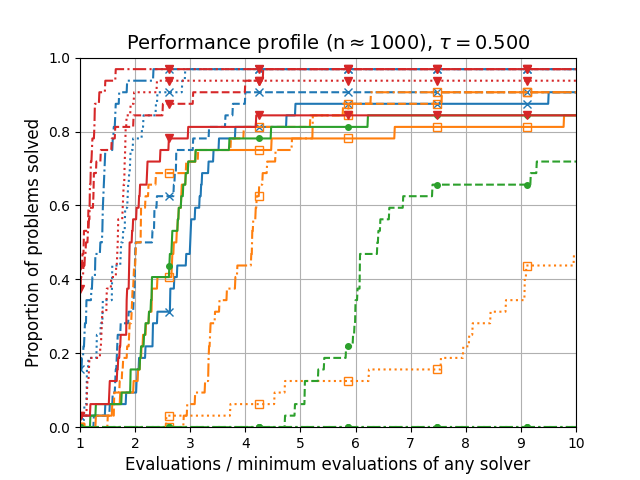}
         \end{subfigure}
         \hfill
         \begin{subfigure}{0.49\textwidth}
             \centering
             \includegraphics[width=\textwidth]{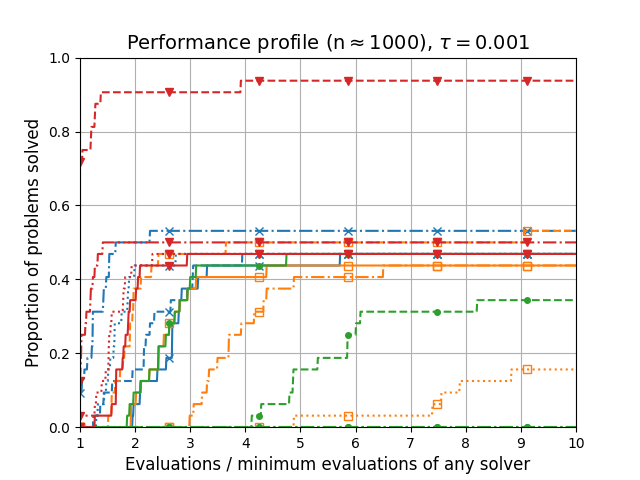}
         \end{subfigure}
         \caption{Performance profiles with $\tau=0.5$ and $\tau=10^{-3}$ comparing the function evaluations.}
         \label{fig:LSPerfwSL05&1e-3evals}
    \end{figure}\FloatBarrier
    
    \begin{figure}[!htb]
         \begin{subfigure}{0.49\textwidth}
             \centering
             \includegraphics[width=\textwidth]{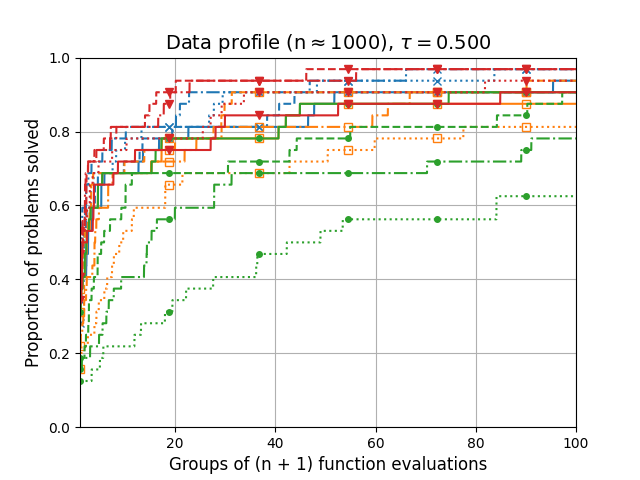}
         \end{subfigure}
         \hfill
         \begin{subfigure}{0.49\textwidth}
             \centering
             \includegraphics[width=\textwidth]{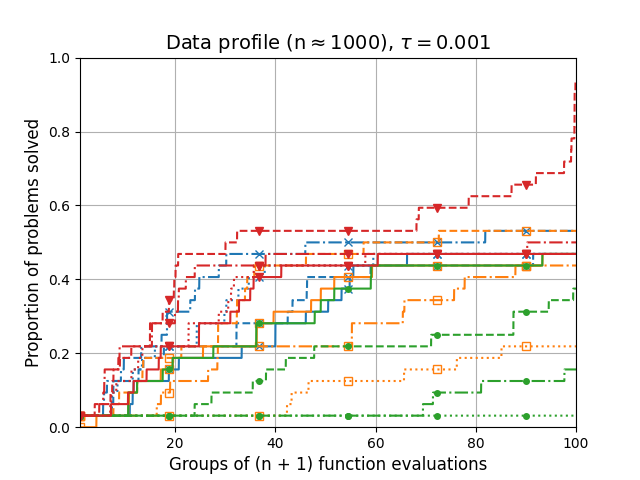}
         \end{subfigure}
         \caption{Data profiles with $\tau=0.5$ and $\tau=10^{-3}$ comparing the function evaluations.}
         \label{fig:LSDatawSL05&1e-3evals}
    \end{figure}\FloatBarrier

	\section{Conclusion}\label{sec:conclusion}
	In this paper, we proposed $\myalg$, a random subspace derivative-free trust-region algorithm for unconstrained deterministic optimization problems.  This algorithm extends the previous random subspace trust-region algorithms to a much broader scope by using quadratic approximation models instead of linear models and getting rid of the requirement that the problems take a certain form (e.g., nonlinear least-squares).  We provided theoretical guarantees for $\myalg$ including model accuracy, sample set geometry, and subspace quality.  In particular, we provided a $Q$-fully quadratic modeling technique that is easy to handle in both theoretical analysis and computer programming aspects.  Based on a framework introduced in~\cite{cartis2023scalable}, we proved the almost-sure global convergence of $\myalg$ and gave an expected result on the number of iterations required to reach a certain accuracy.  Numerical results demonstrated the efficiency of using quadratic models over linear models in $\myalg$ and the benefits of exploring the structure of objective functions when possible.

    In addition, most of the theoretical results in this paper can be applied more broadly to other DFO analyses, as they do not depend on the given algorithm.  In particular, we note that Subsections~\ref{subsec:modelconstruct} and \ref{subsec:modelAcc} provide a new method to construct $Q$-fully quadratic models, which is more accurate than required by the framework in \cite{cartis2023scalable} and the convergence analysis in this paper.  Subsections \ref{subsec:samplesetgeometry} and \ref{subsec:subspaceQual} provide theoretical guarantees for the sample set management procedure and the subspace updating procedure used in this paper.  These results can also be extended to provide theoretical guarantees for some of the procedures used in \cite{cartis2023scalable}.  For example, Theorem \ref{thm:orthisthebest} provides a theoretical explanation for the direction-generating mechanism used in \cite[Algorithm~5]{cartis2023scalable}.  Subsection \ref{subsec:subspaceQual} shows how the choice of $p_{\mathrm{rand}}$, $QR$-factorization, and matrix enhancement influence $\alpha$-well alignedness, which are not clearly explained in \cite{cartis2023scalable}.

\appendix
\section{More numerical results}\label{app:morenumexp}
    \subsection{Results on general unconstrained problems}
    \subsubsection{Results based on runtime}
    \begin{figure}[!htb]
         \begin{subfigure}{0.49\textwidth}
             \centering
             \includegraphics[width=\textwidth]{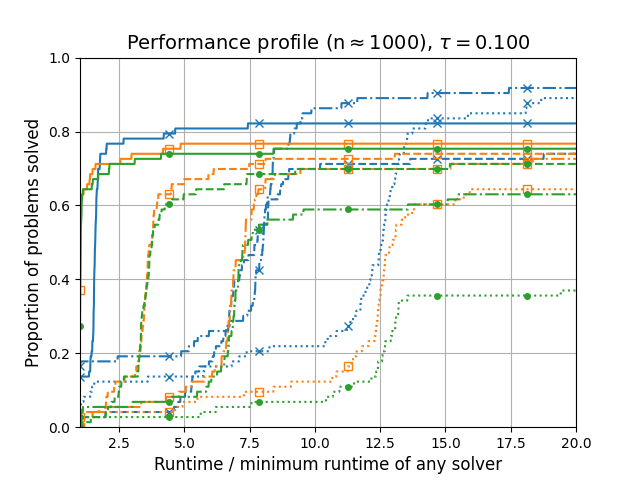}
         \end{subfigure}
         \hfill
         \begin{subfigure}{0.49\textwidth}
             \centering
             \includegraphics[width=\textwidth]{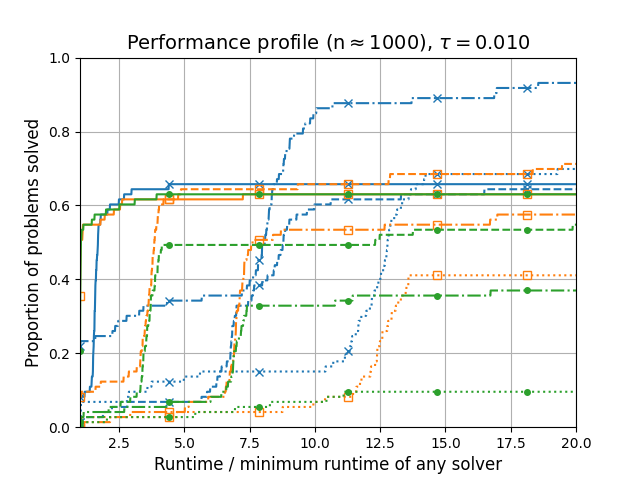}
         \end{subfigure}
         \caption{Performance profiles with $\tau=10^{-1}$ and $\tau=10^{-2}$ comparing the runtime.}
         \label{fig:combinedPerf01&1e-2rt}
    \end{figure}\FloatBarrier
    \begin{figure}[!htb]
         \begin{subfigure}{0.49\textwidth}
             \centering
             \includegraphics[width=\textwidth]{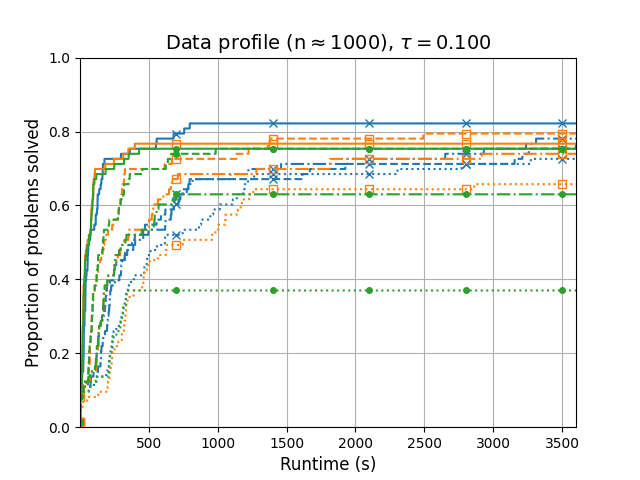}
         \end{subfigure}
         \hfill
         \begin{subfigure}{0.49\textwidth}
             \centering
             \includegraphics[width=\textwidth]{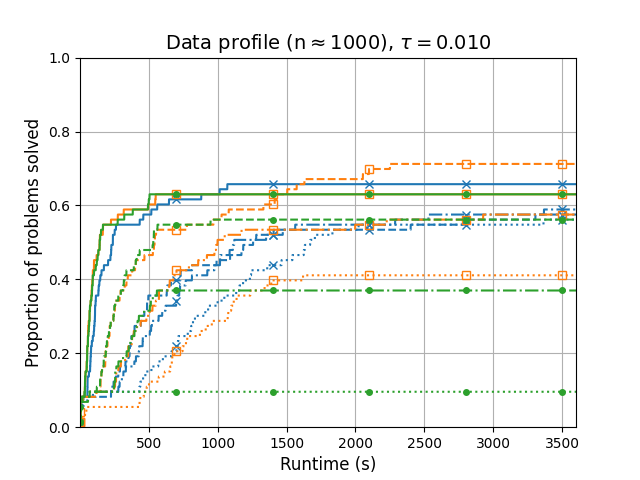}
         \end{subfigure}
         \caption{Data profiles with $\tau=10^{-1}$ and $\tau=10^{-2}$ comparing the runtime.}
         \label{fig:combinedData01&1e-2rt}
    \end{figure}\FloatBarrier

    \subsubsection{Results based on function evaluations}
    \begin{figure}[!htb]
         \begin{subfigure}{0.49\textwidth}
             \centering
             \includegraphics[width=\textwidth]{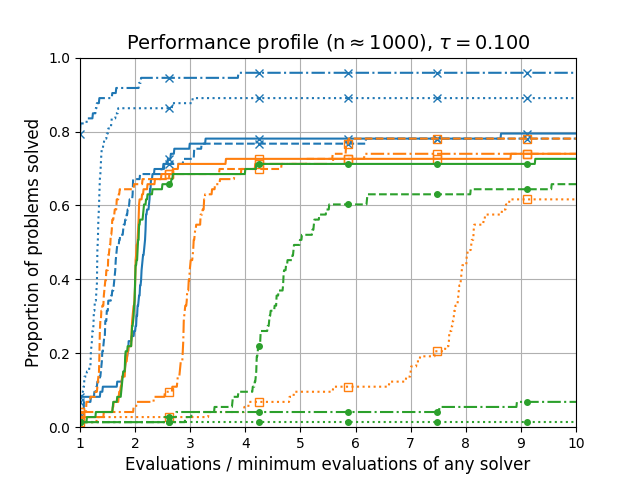}
         \end{subfigure}
         \hfill
         \begin{subfigure}{0.49\textwidth}
             \centering
             \includegraphics[width=\textwidth]{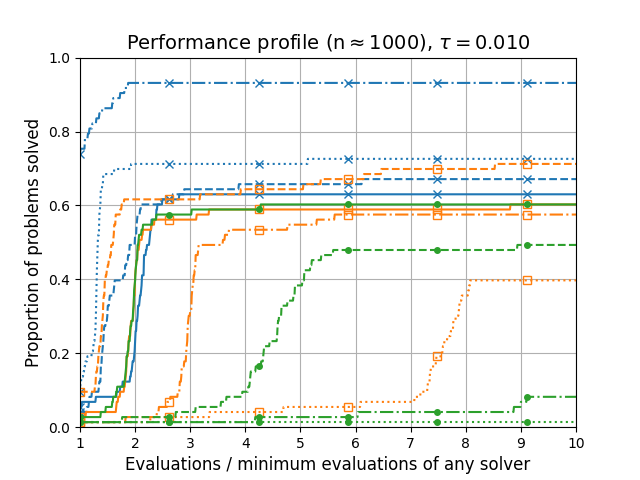}
         \end{subfigure}
         \caption{Performance profiles with $\tau=10^{-1}$ and $\tau=10^{-2}$ comparing the function evaluations.}
         \label{fig:combinedPerf01&1e-2evals}
    \end{figure}\FloatBarrier
    \begin{figure}[!htb]
         \begin{subfigure}{0.49\textwidth}
             \centering
             \includegraphics[width=\textwidth]{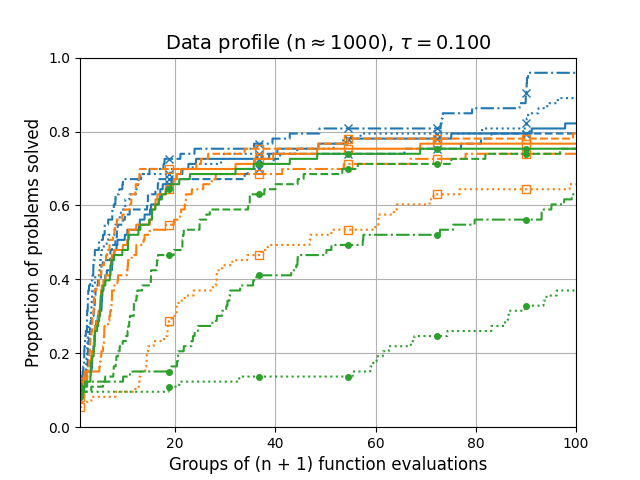}
         \end{subfigure}
         \hfill
         \begin{subfigure}{0.49\textwidth}
             \centering
             \includegraphics[width=\textwidth]{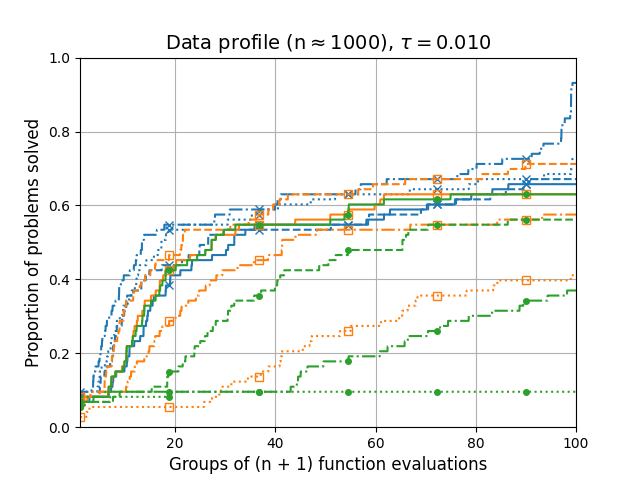}
         \end{subfigure}
         \caption{Data profiles with $\tau=10^{-1}$ and $\tau=10^{-2}$ comparing the function evaluations.}
         \label{fig:combinedData01&1e-2evals}
    \end{figure}\FloatBarrier

    \subsection{Results on nonlinear least-squares problems}
    \subsubsection{Results based on runtime}
    \begin{figure}[!htb]
         \begin{subfigure}{0.49\textwidth}
             \centering
             \includegraphics[width=\textwidth]{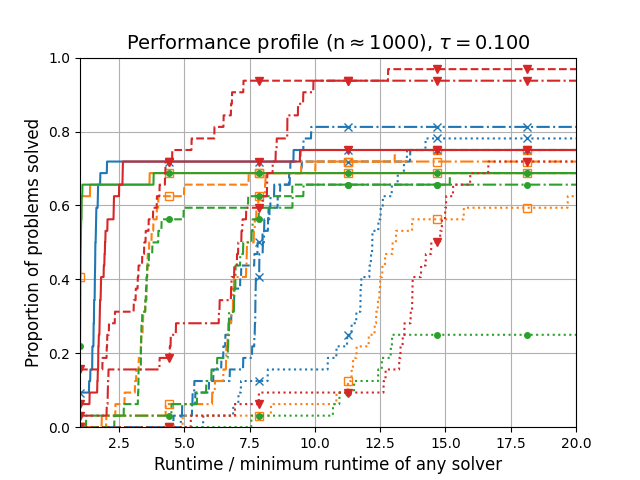}
         \end{subfigure}
         \hfill
         \begin{subfigure}{0.49\textwidth}
             \centering
             \includegraphics[width=\textwidth]{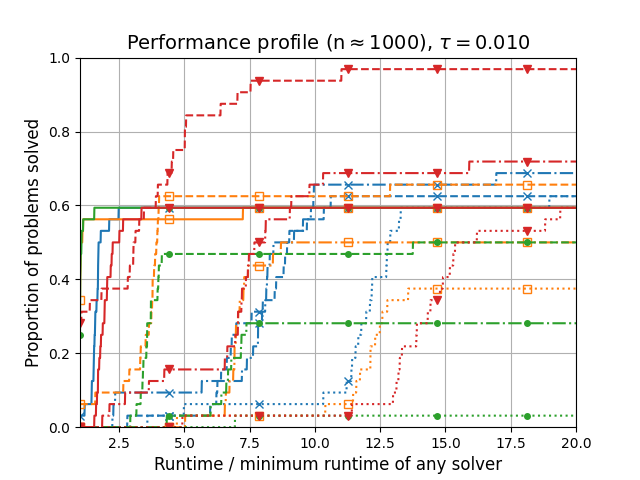}
         \end{subfigure}
         \caption{Performance profiles with $\tau=10^{-1}$ and $\tau=10^{-2}$ comparing the runtime.}
         \label{fig:LSPerfwSL01&1e-2rt}
    \end{figure}\FloatBarrier
    \begin{figure}[!htb]
         \begin{subfigure}{0.49\textwidth}
             \centering
             \includegraphics[width=\textwidth]{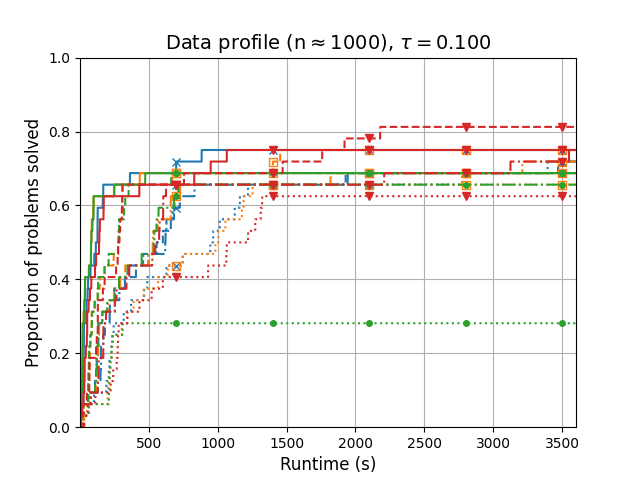}
         \end{subfigure}
         \hfill
         \begin{subfigure}{0.49\textwidth}
             \centering
             \includegraphics[width=\textwidth]{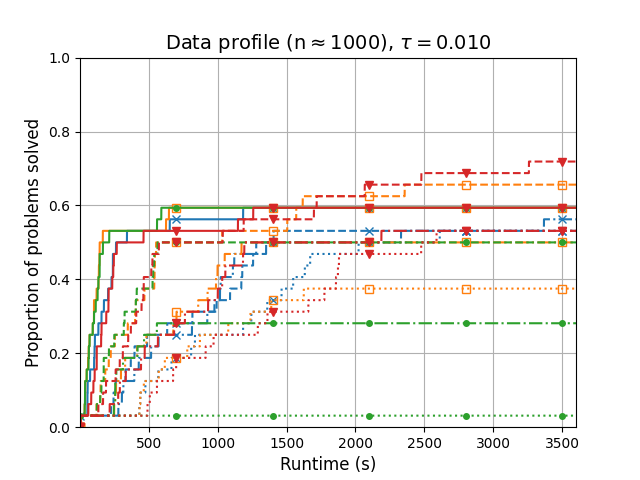}
         \end{subfigure}
         \caption{Data profiles with $\tau=10^{-1}$ and $\tau=10^{-2}$ comparing the runtime.}
         \label{fig:LSDatawSL01&1e-2rt}
    \end{figure}\FloatBarrier

    \subsubsection{Results based on function evaluations}
    \begin{figure}[!htb]
         \begin{subfigure}{0.49\textwidth}
             \centering
             \includegraphics[width=\textwidth]{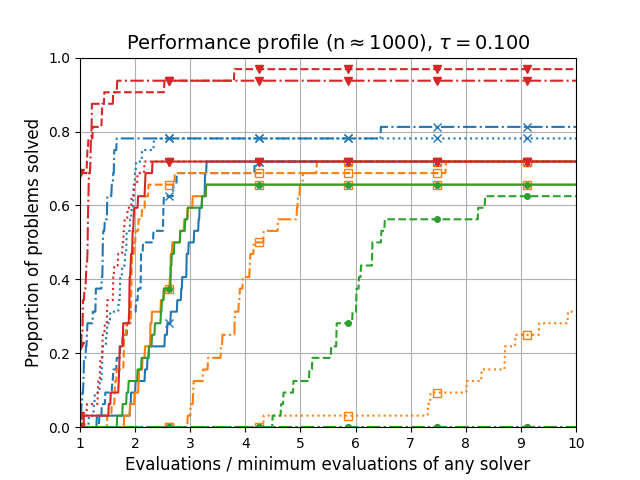}
         \end{subfigure}
         \hfill
         \begin{subfigure}{0.49\textwidth}
             \centering
             \includegraphics[width=\textwidth]{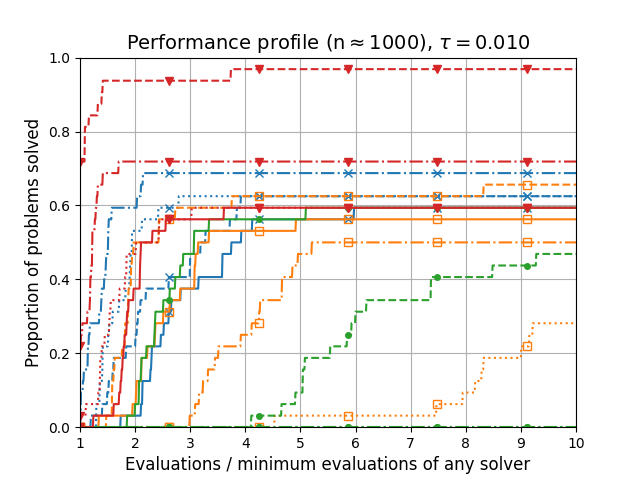}
         \end{subfigure}
         \caption{Performance profiles with $\tau=10^{-1}$ and $\tau=10^{-2}$ comparing the function evaluations.}
         \label{fig:LSPerfwSL01&1e-2evals}
    \end{figure}\FloatBarrier
    \begin{figure}[!htb]
         \begin{subfigure}{0.49\textwidth}
             \centering
             \includegraphics[width=\textwidth]{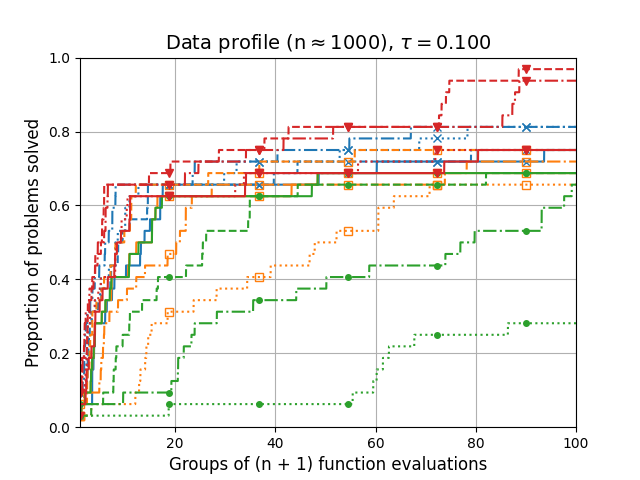}
         \end{subfigure}
         \hfill
         \begin{subfigure}{0.49\textwidth}
             \centering
             \includegraphics[width=\textwidth]{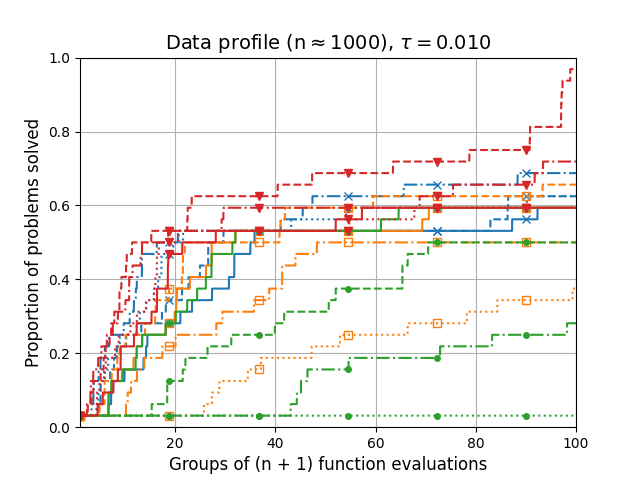}
         \end{subfigure}
         \caption{Data profiles with $\tau=10^{-1}$ and $\tau=10^{-2}$ comparing the function evaluations.}
         \label{fig:LSDatawSL01&1e-2evals}
    \end{figure}\FloatBarrier

\bibliographystyle{siam}
\bibliography{references}
\end{document}